\documentclass[11pt, oneside]{amsart}   
\usepackage{geometry}                	
\usepackage{graphicx}

\usepackage{amsmath,amssymb}
\usepackage{color}
\usepackage{hyperref}

\newtheorem{theorem}{Theorem}[section]
\newtheorem{corollary}[theorem]{Corollary}
\newtheorem{proposition}[theorem]{Proposition}
\newtheorem{lemma}[theorem]{Lemma}
\newtheorem{definition}[theorem]{Definition}
\newtheorem{remark}[theorem]{Remark}

\def\R{\mathbb{R}}

\def\N{\mathbb{N}}

\def\cD{\mathcal{D}}

\def\cP{\mathcal{P}}

\def\aa{{\boldsymbol{a}}}
\def\bb{\boldsymbol{b}}
\def\cc{\boldsymbol{c}}

\def\ee{\boldsymbol{e}}

\def\ii{\boldsymbol{i}}

\def\qq{\boldsymbol{q}}

\def\sb{\boldsymbol{s}}

\def\uu{{\boldsymbol{u}}}
\def\vv{\boldsymbol{v}}

\def\xx{\boldsymbol{x}}
\def\yy{\boldsymbol{y}}

\def\00{\boldsymbol{0}}
\def\11{\boldsymbol{1}}
\def\22{\boldsymbol{2}}
\def\33{\boldsymbol{3}}
\def\44{\boldsymbol{4}}
\def\55{\boldsymbol{5}}
\def\66{\boldsymbol{6}}
\def\77{\boldsymbol{7}}
\def\88{\boldsymbol{8}}
\def\99{\boldsymbol{9}}

\def\aalpha{{\boldsymbol{\alpha}}}

\def\ggamma{{\boldsymbol{\gamma}}}

\def\ssigma{{\boldsymbol{\sigma}}}

\def\mfu{{{\mathfrak{u}}}}
\def\mfU{{{\mathfrak{U}}}}
\def\mfv{{{\mathfrak{v}}}}
\def\mfw{{{\mathfrak{w}}}}
\def\mfD{{{\mathfrak{D}}}}

\DeclareMathOperator*{\spn}{span}

\DeclareMathOperator*{\argmin}{arg\,min}

\def\mix{{\qopname\relax o{mix}}}

\def\mybigtimes{\mathop{\mathchoice{
   \vcenter{\hbox to10bp{\vrule height15bp width0pt \pdfliteral{
   q 1 J .8 w 0 1 m 10 14 l S 0 14 m 10 1 l S Q
}\hss}}}{
   \vcenter{\hbox to10bp{\kern1bp\vrule height10bp width0pt \pdfliteral{
   q 1 J .65 w 0 0 m 8 10 l S 0 10 m 8 0 l S Q
}\hss}}}{\times}{\times}
}}

\def\transpose{{\rm T}}

\newcommand{\dom}{\mathcal{D}}

\title[Constructive Approximation of High-Dimensional
  Functions]
      {Constructive Approximation of High-Dimensional Functions
  with  Small Efficient Dimension with Applications in Uncertainty Quantification}

\author{Christian Rieger}
\address{Christian Rieger,  FB Mathematik und Informatik,
  Philipps-Universität Marburg, 35032 Marburg, Germany}
\email{riegerc@mathematik.uni-marburg.de}
\author{Holger Wendland}
\address{Holger Wendland, Department of Mathematics, University of Bayreuth}
\email{holger.wendland@uni-bayreuth.de}

\date{November 22, 2024}							

\begin{document}
\begin{abstract}
In this paper, we show that the approximation of high-dimensional functions, which are
effectively low-dimensional, does not suffer from the curse of
dimensionality. This is shown first in a general reproducing kernel
Hilbert space set-up and then specifically for Sobolev and
mixed-regularity Sobolev spaces. Finally, efficient estimates are
derived for deciding whether a high-dimensional function is
effectively low-dimensional by studying error bounds in weighted
reproducing kernel Hilbert spaces. The results are applied to
parametric partial differential equations, a typical problem from
uncertainty quantification. 
\end{abstract}

\maketitle

\section{Introduction}
The approximation of high-dimensional functions is still a challenging
task because of the often occurring curse of dimensionality, see
\cite{Bellman-57-1,Bellman-61-1}.  While the
curse of dimensionality in general cannot be beaten, as it is inherent to the
problem of approximating functions from certain classes, see
\cite{Rieger-Wendland-24-1} for an elaborate discussion but also
\cite{Novak-Wozniakowski-08-1,  Novak-Wozniakowski-10-1,
  Novak-Wozniakowski-12-1} for an information-based complexity view point, it is
important to identify relevant classes of high-dimensional functions which do not
suffer from the curse of dimensionality and to develop efficient,
computable methods to approximate such functions.
One way of tackling this problem is to use approximations which are sums of
lower-dimensional functions. This approach is often based on an {\em
  Analysis of Variance (ANOVA)}
or an anchored decomposition of the underlying, unknown function. Such
decompositions were first studied in \cite{Hoeffding-48-1,Sobol-93-1}
and are used in the context of {\em high-dimensional model
  representation}, \cite{Rabitz-Alics-99-1}, in data analysis,
  \cite{Ziehn-Tomlin-09-1}, and machine learning,
  \cite{Bastian-Rabitz-18-1,Li-etal-17-1}. A general unifying
  description, using projections, can be found  in \cite{Kuo-etal-10-1}. 

  In the recent
paper \cite{Rieger-Wendland-24-1}, we have shown that certain
subclasses of Sobolev functions and mixed regularity Sobolev functions
do not suffer from the curse of dimensionality and can be efficiently
approximated using a number of points which grows at most polynomially
in the space dimension. These subclasses of Sobolev functions in
$d$-variables consists, for example, of functions that can be written
as the sum of functions that depend only on $n<d$ variables. The
computational cost of computing an approximation to such a function
grows only polynomially in $d$, showing that such functions can be
approximated efficiently.  If this is possible, we will call $n$ the
{\em order} or {\em effective dimension} of such a function, 
though the term {\em effective dimension} is particularly in the ANOVA
context slightly differently defined, see \cite{Kuo-etal-10-1} and the
literature given therein for a discussion 
of these concepts. 

The goal of this paper is to extend previous results from
\cite{Rieger-Wendland-24-1} in the following way. We will
discuss the problem of approximating a high-dimensional function $f$
which is effectively low-dimensional, i.e. it  can be written in the form
$f=f_1+f_2$, where $f_2$ is assumed to be {\em small} 
and $f_1$ is assumed to be of the above form, i.e. it
is a sum of functions depending only $n$ instead of $d$ variables. 

The challenge here is that while we know how to approximate $f_1$ if we
would know its values at  specific data sets, we cannot measure it since
we can only measure $f$. We will
address this problem of {\em mismeasurement} in Section
\ref{sec:mismeasured} by first looking at the general set-up in
reproducing kernel Hilbert spaces and then specify it to the above
situation. In Section \ref{sec:owen}, we will derive concepts on
deciding quantitatively and qualitatively whether $f$ is effectively
low-dimensional. This section is based on \cite{Owen-19-1} but uses
an anchored rather than ANOVA decompositions. In the final section, we
discuss parametric partial differential equations, a typical problem
often considered in  uncertainty quantification. We will show that
under mild assumption of the parametric partial differential equations, the solutions are
effectively low-dimensional and thus can be approximated efficiently
by the tools developed in this paper.

While we are only discussing anchored decompositions,
there is an inherent connection to ANOVA decompositions. This is a
consequence of the close relation between the anchored and ANOVA
terms, see for example \cite{Gilbert-etal-22-1,Gnewuch-etal-17-1, Hefter-etal-16-1,
  Hinrichs-Schneider-16-1, Kritzer-etal-17-1}.

\subsection{Anchored Decomposition}
The rest of this section is devoted to introducing the necessary
notation and material on anchored decompositions. 
In this paper we are exclusively dealing with
the anchored decomposition of a function. To this end, we assume that
$[\aa,\bb]=[a_1,b_1]\times \cdots [a_d,b_d]$ is a $d$-variate interval. Set
$\mfD:=\{1,\ldots,d\}$ and let $\cP(\mfD)=\{\mfu:
\mfu\subseteq\mfD\}$ be the set of all subsets of $\mfD$.
For a set $\mfu\subseteq\mfD$  we let $\#\mfu$ be the number of elements
in $\mfu$ and set  $\aa_\mfu = (a_j:j\in\mfu)\in\R^{\#\mfu}$.  

\begin{definition}
  \begin{itemize}
 \item A set $\Lambda \subseteq \cP(\mfD)$ of subsets of
   $\mfD=\{1,\ldots,d\}$ is called  {\em downward closed} if for all  
$\mfu\in\Lambda$ and $\mfv\subseteq\mfu$ also
$\mfv\in\Lambda$. 
\item A continuous function  $f:[\aa,\bb] \to\R$
{\em has a  $\Lambda$-representation}, if it can be written as
\begin{equation}\label{flambda}
  f=\sum_{\mfu\in\Lambda} g_\mfu,
\end{equation}
where each $g_\mfu$ is a function depending only on the variables with
indices in $\mfu$. If $\Lambda=\cP(\mfD)$, i.e. if
\begin{equation}\label{fulldecomposition}
  f = \sum_{\mfu\subseteq\mfD} g_\mfu,
\end{equation}
then this is referred to as a {\em full decomposition}.
  \end{itemize}
\end{definition}
Typically, such representations start with a full decomposition,
derived using certain projection methods, and then argue that under
certain assumptions the higher-order terms vanish.

To describe the anchored representation that will be used throughout
this paper, we will use the following notation. For  $\xx, \cc\in [\aa,\bb]$ and
$\mfu\subseteq\mfD$ we let $(\xx;\cc)_\mfu\in [\aa,\bb]\subseteq\R^d$ be the vector 
\[
(\xx;\cc)_\mfu:=\begin{cases}
x_j & \mbox{ if } j\in \mfu \\
c_j & \mbox{ if } j\not\in \mfu.
\end{cases} 
\]
We will abuse this notation also in the case that a vector
$\widetilde{\xx}\in\R^{\#\mfu}$ is given and write $(\widetilde{\xx};\cc)_\mfu\in[\aa,\bb]$
for the vector  having components $x_j$ for $j\in \mfu$ and $c_j$ otherwise. 
We now have the well-known {\em anchored
  decomposition} of functions, which can, for example, be
found in \cite{Kuo-etal-10-1}.

\begin{theorem}\label{thm:decomposition}
Let $\cc\in[\aa,\bb]$ be a fixed point, the {\em anchor}.  Let
$H\subseteq C[\aa,\bb]$ a linear subspace of continuous functions.
Any function $f\in H$ has an {\em anchored
  decomposition}, i.e.   it can be written in the form
\begin{equation}\label{eq:anchor}
	f(\xx)=\sum_{\mfu \subseteq \mfD }f_{\mfu;\cc}(\xx_{\mfu}),
\end{equation}
where the components $f_{\mfu;\cc}$ are functions depending only on
the variables with indices in $\mfu$ and are given by 
\begin{align}
  f_{\emptyset;\cc}&\equiv f(\cc), \label{eq:anova1}\\
  f_{\mfu;\cc}(\xx_{\mfu})&=
  f((\xx;\cc)_{\mfu})-\sum_{\mfv \subsetneq \mfu }f_{\mfv;\cc}(\xx_{\mfv})
  =\sum_{\mfv \subseteq \mfu }(-1)^{\#\mfu -\#\mfv}f ((\xx;\cc)_{\mfv}). \label{eq:anova2}
\end{align}
They satisfy the {\em annihilation property}
$  f_{\mfu;\cc}(\xx_{\mfu})=0 $, whenever there is a 
 $ j\in \mfu$  such that $x_j=c_j $.
Moreover, the following two properties hold.
\begin{itemize}
\item If $f $ has a full decomposition (\ref{fulldecomposition}) where the
  components $g_\mfu$ also satisfy the annihiliation property
  then $g_\mfu = f_{\mfu;\cc}$ for all
  $\mfu\subseteq \mfD$. In this way,  the  decomposition (\ref{eq:anchor}) is unique.

\item If  $f$ has a full decomposition (\ref{fulldecomposition}) 
and if $\mfv\subseteq\mfD$ is a subset such that
    $g_\mfu=0$ for all $\mfu\subseteq \mfD$ with $\mfv\subseteq\mfu$
    then also $f_{\mfu;\cc}=0$ for all such $\mfu$.
\end{itemize}
\end{theorem}
The last property ensures that if there is a $\Lambda$-decomposition
of $f$ with a downward closed $\Lambda$, then also the anchored
decomposition uses only the terms $f_{\mfu;\cc}$ with
$\mfu\in\Lambda$. Thus, in our theoretical investigations later on, we
can restrict ourselves always to anchored decompositions.

If convenient, we will also interpret the functions $f_{\mfu;\cc}$ as
functions on $[\aa,\bb]$ rather than $[\aa_\mfu,\bb_\mfu]$. As these
functions are constant with respect to the variables with indices in
$\mfD\setminus\mfu$, we have in particular
\begin{equation}\label{constantintegral}
\|f_{\mfu;\cc}\|_{L_2([\aa,\bb])}^2 =
\left(\prod_{j\in\mfD\setminus\mfu}(b_j-a_j)\right)
\|f_{\mfu;\cc}\|_{L_2[\aa_\mfu,\bb_\mfu]}^2.
\end{equation}

\section{Approximation of Mismeasured Functions}\label{sec:mismeasured}
 
In this section, we will discuss the following generic approximation
problem. Let $\Omega\subseteq \R^d$ be an arbitrary domain.
Let $H$ be a reproducing kernel Hilbert space of functions
$f:\Omega\to\R$, i.e. there is a unique function
$K:\Omega\times\Omega\to\R$ with $K(\cdot,\xx)\in H$ for all
$\xx\in\Omega$ and with $f(\xx)=\langle f,K(\cdot,\xx)\rangle_H$ for
all $\xx\in \Omega$ and $f\in H$.

Let $H_1\subseteq H$ be a closed subspace of $H$ and let
$H_2=H_1^\bot$ be the orthogonal complement such that we have
$H=H_1\oplus H_2$. This means every $f\in H$ has a unique
decomposition $f=f_1+f_2$ with $f_i\in H_i$, $1\le i\le 2$. 
Given a point set $X=\{\xx_1,\ldots,\xx_N\}\subseteq\Omega$, we want to
compute an approximation to $f\in H$ from $H_1$ using the data
$f(X)=(f(\xx_1),\ldots,f(\xx_N))^\transpose$. We will do this by
computing a regularized regression to $f_1$ using the known data values $f(X)$
instead of the unknown data values $f_1(X)$. Hence, our approach can
either be interpreted as approximating a function $f$ from a certain
subspace $H_1$ to which it does not belong or as approximating a
function $f_1$ from that subspace using the wrong data $(X,f(X))$. The
reasoning behind this is, as pointed out in the introduction, that we
will assume that $f_1$ is the dominating part of $f$.

The precise definition of our approach is as follows.

\begin{definition}
  Let $\lambda>0$ be given. Under the assumptions above, set
  \[
  J(s)=J_{X,\lambda,f(X)}(s):= \sum_{j=1}^N |f(\xx_j)-s(\xx_j)|^2 +
  \lambda\|s\|_H^2, \qquad s\in H.
  \]
Then, the approximation $Q_{X,\lambda} f\in H_1$ to $f\in H$ from $H_1$ using
the data $(X,f(X))$ is defined as 
\[
Q_{X,\lambda} f:=\argmin_{s\in H_1} J_{X,\lambda,f(X)}(s).
\]
\end{definition}
Obviously, a different loss function than the squared absolute value can be used in
the sum but the advantage of using the squared absolute value is that
the approximation process is linear and the approximation can be
computed by solving a linear system. Crucial for this is the
following, well-known result. A proof of the first statement can
already be found in \cite{Aronszajn-50-1}, a proof of the second
statement is in \cite{Wahba-75-1}. In its formulation, we use the
standard notation $K_1(Y,X)=(K_1(\yy_i,\xx_j))$ 
    for any sets $X=\{\xx_1,\ldots,\xx_N\}$ and
    $Y=\{\yy_1,\ldots,\yy_M\}$ and $K_1:\Omega\times\Omega\to\R$.

\begin{lemma}\label{lem:lem1}
  Let $H$ be a reproducing kernel Hilbert space with kernel $K:\Omega\times\Omega\to\R$. 
  \begin{enumerate}
  \item If $H_1\subseteq H$ is a closed subspace then $H_1$ is also a
    reproducing kernel Hilbert space. If $P:H\to H_1$ is the
    orthogonal projection onto $H_1$ then the 
    reproducing kernel of $H_1$ is given by
    $K_1(\cdot,\xx)=PK(\cdot,\xx)$, $\xx\in \Omega$.
  \item The approximation $Q_{X,\lambda} f\in H_1$ to $f\in H$ using the data
    $(X,f(X))$ and $\lambda>0$ is given by
    \[
    Q_{X,\lambda}f = K_1(\cdot,X) \left(K_1(X,X)+\lambda I\right)^{-1}
    f(X).
    \]
  \end{enumerate}
\end{lemma}

To bound the error $f-Q_{X,\lambda} f$, we will split the
error into $f_2=f-f_1$ and $f_1-Q_{X,\lambda }f$, taking the point of view
that $Q_{X,\lambda}f$ is rather an approximation to $f_1$ using the
wrong data as pointed out above.

To bound the second term, we will
employ typical sampling inequalities, see \cite{Arcangeli-etal-07-1,
  Arcangeli-etal-12-1, Rieger-Wendland-17-1,
  Wendland-Rieger-05-1}, though we will need them to hold for the
subspace $H_1$ and not for the whole space $H$, as is typically the
case in the papers cited above.

Nonetheless, for such sampling inequalities we typically need
bounds on $\|f_1-Q_{X,\lambda} f\|_H$ and $\|(f_1-Q_{X,\lambda}
f)(X)\|_p$, where $\|\cdot\|_p$ is the standard $\ell_p$-norm on
$\R^N$, i.e.
\[
\|\yy\|_p = \begin{cases}
  \left(  \sum_{j=1}^N |y_j|^p\right)^{1/p} & \mbox{ if } 1\le p<\infty,\\
  \max_{1\le j\le N} |y_j| & \mbox{ if } p=\infty.
\end{cases}
\]
In the formulation of the following results, we will use that the 
equivalence of $\ell_p$-norms on $\R^N$ is given by
$\|\yy\|_q \le c_{qp} \|\yy\|_p$, $\yy\in\R^N$ with equivalence
constant 
\[
c_{q,p} = \begin{cases}
  N^{\frac{1}{q}-\frac{1}{p}} & \mbox{ if } p\ge q\\
  1 & \mbox{ if } p\le q,
\end{cases}
\]
which we will particularly employ in the form $\widetilde{c}_p\|\yy\|_2\le
\|\yy\|_p\le c_p\|\yy\|_2$, i.e. with  $\widetilde{c}_p=c_{2,p}^{-1}$
and $c_p=c_{p,2}$. Moreover, we will use the notation
$\|f\|_{\ell_p(X)}:=\|f(X)\|_p$. 

\begin{proposition}\label{prop:bounds} Let $H$ be a reproducing kernel
  Hilbert space with kernel 
  $K:\Omega\times\Omega\to\R$ and let  $H_1\subseteq H$ be a closed
  subspace with orthogonal complement $H_2\subseteq H$ and
  with kernel $K_1:\Omega\times\Omega\to\R$.
  Then, for any $f=f_1+f_2\in H$ with $f_i\in H_i$, $i=1,2$, we have
\[
\|Q_{X,\lambda} f\|_H \le \|f\|_H + \frac{1}{\sqrt{\lambda}}
\|f_2\|_{\ell_2(X)}
\]
and, for $1\le p\le \infty$,
\[
\|f_1-Q_{X,\lambda}f\|_{\ell_p(X)} \le c_p 
\left(2 \|f_2\|_{\ell_2(X)} +\sqrt{\lambda}
\|f\|_H\right),
  \]
  where $c_p$ is the possibly $N$-dependent equivalence
  constant from above. 
\end{proposition}
\begin{proof}
  We set $s_1:=Q_{X,\lambda}f\in H_1$ to simplify the notation. Then, using
  $f(\xx_j)-f_1(\xx_j)=f_2(\xx_j)$, we have
  \[
  \lambda\|s_1\|_H^2 \le J(s_1) \le J(f_1) \le \|f_2\|_{\ell_2(X)}^2+\lambda
  \|f_1\|_H^2.
  \]
  Using the monotonicity of the square root and the obvious bound
  $\|f_1\|_H\le \|f\|_H$  yields the first statement. 
  For the second statement, we   proceed as follows. We start with the obvious bound
  \begin{equation}\label{triangle1}
    \|f_1-s_1\|_{\ell_p(X)}  \le  c_p \|f_1-s_1\|_{\ell_2(X)} \le c_p\left(\|f-s_1\|_{\ell_2(X)} +
    \|f_2\|_{\ell_2(X)}\right) 
    \end{equation}
    and then use again 
\[
    \|f-s_1\|_{\ell_2(X)}^2  \le   J(s_1) 
     \le   J(f_1) \le \|f_2\|_{\ell_2(X)}^2 +
    \lambda \|f_1\|_H^2
\]
Monotonicity of the square root together with (\ref{triangle1}) yields
the second statement.
\end{proof}

As mentioned above, we want to apply these estimates in the context of
sampling inequalities. However, in contrast to the above cited
sources, these sampling inequalities have to hold on $H_1$ and not on the
whole space $H$. Before discussing this in two specific applications,
namely the case of standard and the case of mixed regularity Sobolev
spaces, we formulate a general, generic convergence result.

\begin{theorem}\label{thm:genericconv}
Let $H$ be a reproducing kernel Hilbert space of functions
$f:\Omega\to\R$ with reproducing kernel $H$. Let $H_1$ be a closed
subspace with reproducing kernel
$K_1$. Let $1\le p,q\le\infty$. Assume that there is a constant $C>0$ and two 
functions $F_1,F_2:\N\to [0, \infty)$  such that on $H_1$ a {\em
sampling inequality} of the form 
\begin{equation}\label{genericsamplinginequality}
\|f_1\|_{L_q(\Omega)} \le C \left[ F_1(N) \|f_1\|_H + F_2(N)\|f_1\|_{\ell_p(X)}\right], \qquad
f_1\in H_1,
\end{equation}
holds. Then, the error between any $f=f_1+f_2\in H$ and
$Q_{X,\lambda}f$ can be bounded by
\begin{eqnarray*}
\|f-Q_{X,\lambda} f\|_{L_q(\Omega)} &\le&
\|f_2\|_{L_q(\Omega)} + C
\left[2F_1(N)+c_pF_2(N)\sqrt{\lambda}\right]\|f\|_H\\
& & \mbox{} +
  C \left(\frac{F_1(N)}{\sqrt{\lambda}} +2c_pF_2(N)\right) \|f_2\|_{\ell_2(X)}.
\end{eqnarray*}
  \end{theorem}
\begin{proof}
We start with the obvious bound
\[
\|f-Q_{X,\lambda}f\|_{L_q(\Omega)} 
\le \|f_1-Q_{X,\lambda} f\|_{L_q(\Omega)} + \|f_2\|_{L_q(\Omega)}.
\]
Applying the sampling inequality to $f_1-Q_{X,\lambda}f\in
H_1$ yields
\[
\|f_1-Q_{X,\lambda}f\|_{L_q(\Omega)} \le C \left( F_1(N) \|f_1-Q_{X,\lambda}f\|_H
+ F_2(N)\|f_1-Q_{X,\lambda} f\|_{\ell_p(X)}\right).
\]
Then, using $\|f_1-Q_{X,\lambda} f\|_H \le \|f\|_H + \|Q_{X,\lambda}
f\|_H$ and the two bounds from Proposition \ref{prop:bounds} gives
\begin{eqnarray*}
  \|f_1-Q_{X,\lambda} f\|_{L_p(\Omega)} &\le&
  C F_1(N)\left(2\|f\|_H +
  \frac{1}{\sqrt{\lambda}}\|f_2\|_{\ell_2(X)}\right)\\
&&\mbox{}  + C F_2(N) c_p\left(2\|f_2\|_{\ell_2(X)} + \sqrt{\lambda} \|f\|_H\right).
\end{eqnarray*}
Rearranging the terms in the last expression finally leads to the stated
bound.
\end{proof}

\begin{remark}
In general, the function $F_1$ in the above theorem satisfies
$F_1(N)\to 0$ for $N\to\infty$, enforcing convergence. However, the
asymptotic behavior of the function $F_2$ depends on the
chosen situation. For example, in the case of Sobolev functions, see
Lemma \ref{lem:sampling} below, depending on $p,q$, it is either
constant or goes to zero, as well. In the case of mixed regularity
Sobolev spaces, see Theorem \ref{thm:samplingLambdaSobmixed}, $F_2$ usually grows
asymptotically like a fixed power of $\log(N)$.
\end{remark}

The above bound can be further simplified by choosing the smoothing
parameter $\lambda$ as $\sqrt{\lambda}=F_1(N)/F_2(N)$. If we also use
$\|f_2\|_{\ell_2(X)} \le \sqrt{N} \|f_2\|_{\ell_\infty(\Omega)}$ this
then gives the following result.

\begin{corollary} \label{cor:genapproxresult}
  Under the assumptions of Theorem
  \ref{thm:genericconv} and with  $\sqrt{\lambda}:= F_1(N)/F_2(N)$, the
  error between $f=f_1+f_2$ and $Q_{X,\lambda} f$ can be bounded by
  \[
  \|f-Q_{X,\lambda}f\|_{L_q(\Omega)} \le \|f_2\|_{L_q(\Omega)} + 
  2C(1+c_p) \left[F_1(N) \|f\|_H + F_2(N)
    \|f_2\|_{\ell_2(X)}\right].
  \]
  In particular, for $p=q=\infty$ this yields the bound
  \[
  \|f-Q_{X,\lambda} f\|_{L_\infty(\Omega)} \le \widetilde{C}\left[
    F_1(N) \|f\|_H + \sqrt{N} F_2(N)\|f_2\|_{L_\infty(\Omega)}\right].
    \]
\end{corollary}

The last bound in this corollary shows that in approximating $f\in H$,
the choice of the sub-space $H_1$ might depend on the desired accuracy
and the desired number of samples. To be more precise, the general procedure
for approximating  a function $f\in H$ would theoretically be as follows.
\begin{enumerate}
\item Choose the sample size $N\in \N$ such that $\widetilde{C}
  F_1(N)\|f\|_H<\epsilon/2$.
\item Choose $H_1$ such that
  $\widetilde{C}\sqrt{N}F_2(N)\|f_2\|_{\ell_\infty(\Omega)}\le
  \epsilon/2$.
\end{enumerate}
We will later see, that in some cases it is possible to reverse the
order, i.e. to first choose $H_1$ independently of $N$ and then choose
the sampling set. However, we will also see that it might sometimes be
necessary to choose both connectedly.

To apply this generic result, we need sampling inequalities to derive
error bounds. To this end, we have to 
specify the underlying reproducing kernel Hilbert space $H$. We will
do this for standard and mixed regularity Sobolev spaces. While we
will directly apply the result to mixed regularity Sobolev spaces, we
will use a modification of the above result in the case of standard
Sobolev spaces to avoid the $\sqrt{N}$ term in Corollary
\ref{cor:genapproxresult} altogether.

\subsection{Application in Sobolev Spaces}

In this first application, the reproducing kernel Hilbert space will
be $H^\sigma(\Omega)$ with $\sigma>d/2$. As usual, $H^k(\Omega)$ consists of 
all $f\in L_2(\Omega)$ having weak derivatives $D^\aalpha u\in L_2(\Omega)$
for all $\aalpha\in\N_0^d$ with $|\aalpha|=\|\aalpha\|_1\le k$. The
norm on this space is defined in the usual way by
\[
\|f\|_{H^k(\Omega)}^2 := \sum_{\ell =0}^k |f|_{H^\ell(\Omega)}^2, \qquad
|f|_{H^\ell(\Omega)}^2 :=\sum_{|\aalpha|=\ell} \|D^\aalpha
f\|_{L_2(\Omega)}^2.
\]
If $\Omega$ is a bounded domain with
Lipschitz boundary or $\Omega=\R^d$, this definition can be extended
to also define fractional-order Sobolev spaces $H^\sigma(\Omega)$ with 
$\sigma\in[0,\infty)$. Finally, it is well-known that by the Sobolev
 embedding theorem, $H^\sigma(\Omega)$ is a reproducing kernel Hilbert
space whenever $\sigma>d/2$ and that a kernel can be constructed by restricting a
reproducing kernel of $H^\sigma(\R^d)$ to $\Omega$ if an equivalent
norm is used. For
low-dimensional domains $\Omega\subseteq\R^d$, sampling inequalities
are usually stated using the fill distance $h_{X,\Omega}$ of a point
set $X=\{\xx_1,\ldots,\xx_N\}\subseteq\Omega$, which we recall now
together with the separation radius $q_X$ as
\[
h_{X,\Omega}:=\sup_{\xx\in\Omega}\min_{\xx_j\in X} \|\xx-\xx_j\|_2,
\qquad
q_X:= \frac{1}{2}\min_{j\ne k} \|\xx_j-\xx_k\|_2.
\]
The {mesh ratio} $\rho_{X,\Omega}:=h_{X,\Omega}/q_X$ measures how well
and how economically 
the data sites cover the region $\Omega$. 
We will require the following sampling inequality from \cite{Arcangeli-etal-07-1, 
  Arcangeli-etal-12-1, Rieger-Wendland-17-1}.

\begin{lemma}\label{lem:sampling}
Let $\Omega\subseteq\R^d$ be a bounded domain with a Lipschitz
boundary. Let $p,q\in [1,\infty]$. Let $\tau>d/2$ and
$\gamma:=\max\{2,p,q\}$. Then, there exist constants $h_0>0$
(depending on $\Omega$ and $\tau$) and $C>0$ (depending on $\Omega$,
$\tau$ and $q$) such that for all $X=\{\xx_1,\ldots,\xx_N\}\subseteq
\Omega$ with $h=h_{X,\Omega}\le h_0$ and all $f\in H^\tau(\Omega)$, we have
\[
\|f\|_{L_q(\Omega)} \le C\left(h^{\tau-d(1/2-1/q)_+} |f|_{H^\tau(\Omega)} +
h^{d/\gamma} \|f\|_{\ell_p(X)}\right).
\]
\end{lemma}  

We are particularly interested in sub-spaces $H_1$ of
$H=H^\sigma(\Omega)$, which contain functions having a
$\Lambda$-representation (\ref{flambda}). To this end, we will
restrict ourselves to $\Omega=[\aa,\bb]\subseteq\R^d$ and use the following
result from \cite{Rieger-Wendland-24-1}.

\begin{lemma}
Let $\Lambda\subseteq \cP(\mfD)$ be a downward closed set of subsets of
  $\mfD=\{1,\ldots,d\}$. Let $[\aa,\bb]\subseteq\R^d$ be a closed
interval and let $\sigma>d/2$. Let
$H^\sigma_\Lambda([\aa,\bb])$ be the set of all functions $f\in 
  H^\sigma([\aa,\bb])$ having a $\Lambda$-representation
  (\ref{flambda}). Then, $H^\sigma_\Lambda([\aa,\bb])$  is a closed
  sub-space of $H^\sigma([\aa,\bb])$. 
\end{lemma}

By Theorem \ref{thm:decomposition}, any function $f\in
H_\Lambda^\sigma([\aa,\bb])$ also has an anchored
$\Lambda$-decomposition. 

To apply the general theory from above, we need a sampling inequality
for $H^\sigma_\Lambda([\aa,\bb])$. We will use a generalization
of \cite[Theorem 4.5]{Rieger-Wendland-24-1}, extending that result to
other $L_p$-norms. In contrast to standard sampling inequalities for
Sobolev spaces, a sampling inequality for $H_\Lambda^\sigma([\aa,\bb])$
only holds for specific point sets, exploiting the anchored structure
of the functions for which it should hold.

\begin{definition}\label{def:pointset}
Let $\Lambda\subseteq \cP(\mfD)$ be downward closed. For each $\mfu\in
\Lambda$ choose a point set 
$
\widetilde{X}_{\mfu} = \left\{\widetilde{\xx}_1^{(\mfu)}, \ldots,
\widetilde{\xx}_{N_{\mfu}}^{(\mfu)}\right\}\subseteq[\aa_\mfu,\bb_\mfu]
$ and, using an anchor $\cc\in[\aa,\bb]$,  extend the points of this
  set to obtain the anchored set
$
X_{\mfu} = \left\{(\widetilde{\xx};\cc)_\mfu : \widetilde{\xx}\in \widetilde{X}_{\mfu}\right\}$
Then, a {\em sampling point set for Sobolev $\Lambda$-functions in
  $\Omega\subseteq\R^d$} is given by
\[
X_\Lambda^{(d)}=\bigcup_{\mfu\in\Lambda} X_\mfu,
\]
its {\em fill distance} is defined by
\[
h_{X_\Lambda^{(d)}}:= \max_{\mfu\in\Lambda}
h_{\widetilde{X}_\mfu,[\aa_\mfu,\bb_\mfu]}.
\]
\end{definition}
The following result is the required generalization of the sampling
inequality from \cite{Rieger-Wendland-24-1}.

\begin{theorem}\label{thm:samplingSobolevNew}
Let $[\aa,\bb]\subseteq\R^d$. Let $\Lambda\subseteq\cP(\mfD)$ be
downwards closed. Let $\sigma>d/2$ and $1\le p,q\le \infty$. Let
$\alpha:=\max\{2,q\}$ and $\gamma:=\max\{2,p,q\}$. Then, there are constants
$h_0,C>0$ such that 
\begin{equation}\label{sobolevbnd1}
\|f\|_{L_q([\aa,\bb])} \le C\sum_{\mfv \in \Lambda}\left(
h_{\widetilde{X}_\mfv,[\aa_\mfv,\bb_\mfv]}^{\sigma-\frac{d}{2}+\frac{\#\mfv}{\alpha}}
\|f\|_{H^\sigma([\aa,\bb])} 
+ h_{\widetilde{X}_\mfv,[\aa_\mfv,\bb_\mfv]}^{\frac{\#\mfv}{\gamma}}\|f\|_{\ell_p(X_\mfv)}\right)
\end{equation}
for all $f\in H^\sigma_\Lambda([\aa,\bb])$ and all
sampling point sets $X_\Lambda^{(d)}\subseteq [\aa,\bb]$ with
$h_{X_\Lambda^{(d)}}\le h_0$.
If the low-dimensional sets $\widetilde{X}_\mfv$
are chosen quasi-uniformly, i.e. if there is a constant $c_\rho>0$
such that $\rho_{\widetilde{X}_\mfu ,[\aa_\mfu,\bb_\mfu]}\le c_\rho$
for all $\mfu\in \Lambda$ then
\begin{equation}\label{sobolevbnd2}
\|f\|_{L_2([\aa,\bb])} \le C \left(
\sum_{\mfv\in\Lambda}
h_{\widetilde{X}_\mfv,[\aa_\mfv,\bb_\mfv]}^{\sigma-d/2+\#\mfv/2}
\|f\|_{H^\sigma([\aa,\bb])} +
\left(\sum_{\mfv\in \Lambda} \frac{1}{\# X_\mfv} \|f\|_{\ell_2(X_\mfv)}^2\right)^{1/2}\right).
\end{equation}
The constant $C>0$ depends on $\aa,\bb$ and the cardinality of $\Lambda$.
\end{theorem}

\begin{proof}
We will essentially follow the proof of \cite[Theorem
  4.5]{Rieger-Wendland-24-1} but need a more general sampling
inequality in the low-dimensional setting. Starting with the anchored
representation of any $f\in H^\sigma_\Lambda([\aa,\bb])$, i.e. with
\begin{equation}\label{decomp1}
f = \sum_{\mfu\in\Lambda}\sum_{\mfv\subseteq\mfu} (-1)^{\#\mfu-\#\mfv}
f((\cdot;\cc)_\mfv)
\end{equation}
and recall from \cite{Rieger-Wendland-24-1} the following
properties. First, the components  of  $f\in H^\sigma_\Lambda([\aa,\bb])$
satisfy $f((\cdot;\cc)_\mfv)\in H^{\tau}([\aa_\mfv,\bb_\mfv])$ with
$\tau=\sigma-(d-\#\mfv)/2$. Second, there is a constant
$C_\mfv=C_\mfv(\aa,\bb)$ such that
\(
\|f((\cdot;\cc)_\mfv)\|_{H^\tau([\aa_\mfv,\bb_\mfv])} \le C_\mfv
\|f\|_{H^\sigma([\aa,\bb])}.
\)
Third, we obviously have $\|f((\cdot;\cc)_\mfv)\|_{\ell_p(\widetilde{X}_\mfv)} =
\|f\|_{\ell_p(X_\mfv)}\le \|f\|_{\ell_p(X_\Lambda^{(d)})}$ and 
\[
\|f((\cdot;\cc)_\mfv)\|_{L_p([\aa,\bb])} =
c_p(\aa,\bb,\mfv)\|f((\cdot;\cc)_\mfv)\|_{L_p([\aa_\mfv,\bb_\mfv])},
\]
where the constant is given by
\[
c_p(\aa,\bb,\mfv) = \begin{cases} \prod_{j\in {\mfD\setminus
      \mfv}} (b_j-a_j)^{1/p} & \mbox{ for } 1\le p<\infty,\\
  1 & \mbox{ for } p=\infty.
\end{cases}
\]
With these properties at hand, we can apply the sampling
inequality from Lemma \ref{lem:sampling} to the components to derive
\begin{eqnarray*}
\|f((\cdot;\xx)_\mfv)\|_{L_q([\aa,\bb])}\! &\le&\! C_\mfv\left(
h_\mfv^{\tau -
  \#\mfv\left(\frac{1}{2}-\frac{1}{q}\right)_{+}}
  \|f((\cdot;\cc)_\mfv)\|_{H^\tau([\aa_\mfv,\bb_\mfv])} +
  h_\mfv^\frac{\#\mfv}{\gamma}
  \|f((\cdot;\cc)_\mfv)\|_{\ell_p(\widetilde{X}_\mfv)}\right)\\
  &\le& C_\mfv \left(
h_\mfv^{\sigma-\frac{d}{2}+\frac{\#\mfv}{\max\{2,q\}}}
  \|f\|_{H^\sigma([\aa,\bb])} +
  h_\mfv^\frac{\#\mfv}{\gamma}
  \|f\|_{\ell_p(X_\mfv)}\right),
\end{eqnarray*}
  where we have set
  $h_\mfv:=h_{\widetilde{X}_\mfv,[\aa_{\mfv},\bb_\mfv]}$. Inserting
  the latter bound into (\ref{decomp1}) yields the first error
  estimate in (\ref{sobolevbnd1}). For the second estimate
  (\ref{sobolevbnd2}), we set $p=q=2$ in (\ref{sobolevbnd1}), yielding
  \[
  \|f\|_{L_2([\aa,\bb])} \le C\sum_{\mfv \in \Lambda}\left(
h_\mfv^{\sigma-\frac{d}{2}+\frac{\#\mfv}{2}}\|f\|_{H^\sigma([\aa,\bb])}
+ h_\mfv^{\frac{\#\mfv}{2}}\|f\|_{\ell_2(X_\mfv)}\right).
\]
As the low-dimensional point sets $\widetilde{X}_\mfv$ are
quasi-uniform, we have $\# \widetilde{X}_{\mfv}=\#X_\mfv \le
Ch_\mfv^{-\#\mfv} \le \widetilde{C} \# X_\mfv$
with  constants independent of $\mfv$. This, together with the 
Cauchy-Schwarz inequality, shows 
\[
\sum_{\mfv\in \Lambda}
h_\mfv^{\#\mfv/2} \|f\|_{\ell_2(X_\mfv)} \le
C \sqrt{\#\Lambda} \left(\sum_{\mfv\in \Lambda} \frac{1}{\# X_\mfv}
\|f\|_{\ell_2(X_\mfv)}^2\right)^{1/2},
\]
which then immediately gives the second error bound (\ref{sobolevbnd2}).
\end{proof}

With this, particularly with the second bound (\ref{sobolevbnd2}), we
are able to derive an approximation result for $f\in H^\sigma(\Omega)$
with anticipated small $f_2=f-f_1$, where $f_1\in H^\sigma_\Lambda(\Omega)$.
This result will be a variation of the general result in Corollary
\ref{cor:genapproxresult} with the advantage of avoiding the $\sqrt{N}$
term in front of the discrete $\ell_2$-norm of $f_2$. To formulate it,
we need to modify the approximation operator $Q_{X,\lambda}$ as
follows.

\begin{definition}
Let $\sigma>d/2$ be given. Let $\Lambda\subseteq\cP(\mfD)$ be downward
closed and let 
$X:=X_\Lambda^{(d)}=\cup_{\mfu\in \Lambda} X_\mfu
\subseteq[\aa,\bb]\subseteq\R^d$ be a sampling point 
set for Sobolev $\Lambda$-functions.  For $\lambda>0$  set
\[
J^{(\Lambda)}(s) = J_{X,\lambda,
  f(X)}^{(\Lambda)}:= \sum_{\mfv\in \Lambda}
\frac{1}{\#X_\mfv}\|f-s\|_{\ell_2(X_\mfv)}^2 + \lambda
\|s\|_{H^\sigma([\aa,\bb])}^2, \qquad s\in H^\sigma([\aa,\bb]).
\]
Then, the approximation
$Q^{(\Lambda)}f=Q_{X,\lambda}^{(\Lambda)} f$ is defined as
\[
Q^{(\Lambda)} f:=\argmin_{s\in H_\Lambda^\sigma([\aa,\bb])}
J^{(\Lambda)}(s).
\]
\end{definition}

The operator $Q^{(\Lambda)}$ is indeed  well-defined and linear,
which can be seen as follows. We enumerate
$\Lambda=\{\mfu_1,\ldots,\mfu_Q\}$, set $X=X_\Lambda^{(d)}$, $N=\#X$,
$N_j=\#X_{\mfu_j}$.
and  assume without restriction that $X_\mfu\cap X_\mfv=\emptyset$ for
$\mfu\ne \mfv$.
Then, we can define the block diagonal matrix
$W\in\R^{N\times N}$ by
\[
W = \begin{pmatrix} \frac{1}{N_1} I_{N_1} & & 0\\
  & \ddots & \\
0  & & \frac{1}{N_Q} I_{N_Q}
\end{pmatrix},
\]
where $I_{N_j}\in\R^{N_j\times N_j}$ is the identity matrix,
and rewrite the functional $J^{(\Lambda)}$ as
\[
J^{(\Lambda)} (s) = \left(f(X)-s(X)\right)^\transpose W
\left(f(X)-s(X)\right) + \lambda \|s\|_{H^\sigma([a,b])}^2.
\]
The fact that $Q^{(\Lambda)}$ is well-defined then follows from the
following generalization of the results in Lemma \ref{lem:lem1}. We
include its proof for the convenience of the reader but also refer to
\cite{deBoor-01-2, Kersey-03-1,Reinsch-67-1} for generalizations of
the smoothing spline approach.

\begin{lemma}
Let $H$ be a reproducing kernel Hilbert space with positive definite
kernel $K:\Omega\times\Omega\to\R$   
Let $X=\{\xx_1,\ldots,\xx_N\}\subseteq\Omega$ and
$f(X)=(f(\xx_j))\in\R^N$ for $f\in H$. Given a symmetric and positive definite matrix
$W\in\R^{N\times N}$, the functional
\[
J(s):=(f(X)-s(X))^\transpose W (f(X)-s(X)) + \lambda \|s\|_H^2
\]
has a unique minimum $s^*\in H$, which can be written in the form
$s^*=\aalpha^\transpose K(X,\cdot)$ where the coefficient
$\aalpha\in\R^N$ is given by
\[
\aalpha= (K(X,X)+\lambda W^{-1})^{-1} f(X).
\]
\end{lemma}
\begin{proof}
  The proof is essentially the same as the proof of the unweighted
  problem. First, one shows that the minimum of $J$ has to be attained
  in $V_X=\spn\{K(\cdot,\xx) : \xx\in X\}$ by decomposing a general
  $s\in H$ into $s=s_X+s_X^\bot$ with $s_X\in V_X$ and $s_X^\bot \bot
  V_X$ and using the fact that $s_X^\bot|X =0$. Then, using the
  representation $s=K(X,\cdot)\aalpha$ with $\aalpha\in\R^N$, it is
  straight-forward to see that
  \begin{eqnarray*}
  J(s) &=& f(X)^\transpose W f(X) - 2 \aalpha^\transpose K(X,X)W f(X) +
  \aalpha^\transpose K(X,X)WK(X,X)\aalpha\\
  & & \mbox{} + \lambda \aalpha^\transpose K(X,X)\aalpha.
  \end{eqnarray*}
  whose gradient with respect to $\aalpha$ is
  \[
  \nabla_\aalpha J(s) = 2K(X,X)W\left[-f(X)+K(X,X)+\lambda W^{-1}\right]\aalpha.
  \]
  Setting this to zero and using that with $W$ and $K(X,X)$ also
  $K(X,X)+\lambda W^{-1}$ is positive definite, the statement follows.
\end{proof}

After this, we are able to show that any function $f=f_1+f_2\in
H^\sigma([\aa,\bb])$ with $f_1\in H_\Lambda^\sigma([\aa,\bb])$ and
presumed small $f_2=f-f_1$ can well be approximated with the above
operator $Q^{(\Lambda)}_{X,\lambda}$.

\begin{theorem}\label{thm:sampdecompsob}
Let $\sigma>d/2$ be given. Let $\Lambda\subseteq\cP(\mfD)$ be downward
closed. Let $f_1\in H^\sigma_\Lambda([\aa,\bb])$ be the orthogonal
projection of  $f\in H^\sigma([\aa,\bb])$ and let $f_2=f-f_1$.
Let $X=X_\Lambda^{(d)}$ and assume that the low-dimensional data sets
$\widetilde{X}_{\mfv}$ are quasi-uniform. If their fill distances
$h_\mfv=h_{\widetilde{X}_\mfv,[\aa_\mfv,\bb_\mfv]}$, $\mfv\in\Lambda$,
are sufficiently small then
\begin{eqnarray*}
  \|f-Q_{X,\lambda}^{(\Lambda)} f\|_{L_2([\aa,\bb])} & \le &
  C\left(\sum_{\mfv\in\Lambda} h_{\mfv}^{\sigma-d/2+\#\mfv/2} +
\sqrt{\lambda}\right)\|f\|_H \\
& & \mbox{} + C \left(\frac{1}{\sqrt{\lambda}} \sum_{\mfv\in\Lambda}
h_{\mfv}^{\sigma-d/2+\#\mfv/2} + 1\right)\|f_2\|_{\ell_\infty([\aa,\bb])}.
\end{eqnarray*}
Thus, with $h_{X_\Lambda^{(d)}}=\max_{\mfv\in \Lambda} h_\mfv$, the choices
 \[
  \sqrt{\lambda} = \sum_{\mfv\in\Lambda}
  h_{\mfv}^{\sigma-d/2+\#\mfv/2} \quad \mbox{ or } \quad
  \sqrt{\lambda} = h_{X_\Lambda^{(d)}}^{\sigma-d/2}
  \]
  of the smoothing parameter yields the error bound
  \begin{equation}\label{funderror1}
  \|f- Q_{X,\lambda}^{(\Lambda)}f\|_{L_2([\aa,\bb])}
  \le C\left(h_{X_\Lambda^{(d)}}^{\sigma-d/2}
  \|f\|_H + \|f_2\|_{L_\infty([\aa,\bb])}\right).
  \end{equation}
\end{theorem}
\begin{proof}
To simplify the notation we will use
$Q^{(\Lambda)}=Q^{(\Lambda)}_{X,\lambda}$ and
$J^{(\Lambda)}=J^{(\Lambda)}_{X,\lambda,f(X)}$ as before and set
$H:=H^\sigma([\aa,\bb])$. 
We start with splitting the error as before as
\begin{eqnarray}
\|f-Q^{(\Lambda)} f\|_{L_2([\aa,\bb])} &\le&
\|f_1-Q^{(\Lambda)}f\|_{L_2([\aa,\bb])} + \|f_2\|_{L_2([\aa,\bb])} \nonumber
\\
& \le &
\|f_1-Q^{(\Lambda)}f\|_{L_2([\aa,\bb])} + \prod_{j=1}^d
(b_j-a_j)^{1/2} \|f_2\|_{L_\infty([\aa,\bb])}. \label{dimconstant}
\end{eqnarray}
Next, we proceed with applying the sampling inequality (\ref{sobolevbnd2}) of
Theorem \ref{thm:samplingSobolevNew} to the error
$f_1-Q^{(\Lambda)}f$, yielding, with $h_{\mfv} = h_{\widetilde{X}_\mfv,[\aa_\mfv,\bb_\mfv]}$,
\begin{eqnarray*}
\|f_1-Q^{(\Lambda)} f\|_{L_2([\aa,\bb])} &\le& C 
\sum_{\mfv\in\Lambda}
h_\mfv^{\sigma-d/2+\#\mfv/2} \|f_1-Q^{(\Lambda)} f\|_H\\
& & \mbox{} +C
\left(\sum_{\mfv\in \Lambda} \frac{1}{\# X_\mfv}
\|f_1-Q^{(\Lambda)}f\|_{\ell_2(X_\mfv)}^2\right)^{1/2}.
\end{eqnarray*}
To further bound this, we need to bound the two expressions on the
right-hand side. This is done by 
modifying the bounds in Proposition \ref{prop:bounds}
appropriately to serve our new approximation operator $Q^{(\Lambda)}$.
First, we obviously have
\[
\lambda \|Q^{(\Lambda)} f\|_H^2 \le J^{(\Lambda)}(f_1) = \sum_{\mfv\in
  \Lambda} \frac{1}{\#X_{\mfv}} \|f_2\|_{\ell_2(X_\mfv)}^2 + \lambda
\|f\|_{H}^2,
\]
showing
\begin{equation}\label{sobbd1}
\|f_1-Q^{(\Lambda)} f\|_H \le \frac{1}{\sqrt{\lambda}}
\left(\sum_{\mfv\in\Lambda}\frac{1}{\#X_{\mfv}}
\|f_2\|_{\ell_2(X_{\mfv})}^2\right)^{1/2} + 2 \|f\|_H.
\end{equation}
Next, the other term in the bound above can further be
bounded as follows:
\begin{eqnarray*}
  \sum_{\mfv\in\Lambda} \frac{1}{\# X_\mfv}
  \|f_1-Q^{(\Lambda)}f\|_{\ell_2(X_{\mfv})}^2 &\le &
  2\sum_{\mfv\in\Lambda} \frac{1}{\#X_\mfv}
  \left(\|f-Q^{(\Lambda)}f\|_{\ell_2(X_\mfv)}^2 +
  \|f_2\|_{\ell_2(X_\mfv)}^2\right)\\
  &\le & 2 \sum_{\mfv\in \Lambda}\frac{1}{\#X_{\mfv}}
  \|f_2\|_{\ell_2(X_\mfv)}^2 + 2 J^{(\Lambda)} (f_1)\\
  &\le & 4\sum_{\mfv\in \Lambda} \frac{1}{\# X_\mfv}
  \|f_2\|_{\ell_2(X_\mfv)}^2 + 2 \lambda \|f\|_H^2
\end{eqnarray*}
showing
\begin{equation}\label{sobbd2}
\left(\sum_{\mfv \in \Lambda} \frac{1}{\#X_\mfv}
  \|f_1-Q^{(\Lambda)}f\|_{\ell_2(X_\mfv)}^2\right)^{1/2}  \le 
  \left(\sum_{\mfv\in\Lambda} \frac{1}{\#X_\mfv}
  \|f_2\|_{\ell_2(X_\mfv)}^2 + \sqrt{\lambda}\|f\|_H\right).
\end{equation}
Plugging (\ref{sobbd1}) and (\ref{sobbd2}) into the above bound
finally yields
\begin{eqnarray*}
\|f_1-Q^{(\Lambda)} f\|_{L_2([\aa,\bb])} &\le& C 
\sum_{\mfv\in\Lambda}
h_\mfv^{\sigma-d/2+\#\mfv/2}\left(\frac{1}{\sqrt{\lambda}}
\left(\sum_{\mfv\in\Lambda}\frac{1}{\#X_{\mfv}}
\|f_2\|_{\ell_2(X_{\mfv})}^2\right)^{1/2} + 2 \|f\|_H\right)\\
& & \mbox{} +C
\left(\sum_{\mfv\in\Lambda} \frac{1}{\#X_\mfv}
\|f_2\|_{\ell_2(X_\mfv)}^2 + \sqrt{\lambda}\|f\|_H\right)\\
& = & C\left(\sum_{\mfv\in\Lambda} h_{\mfv}^{\sigma-d/2+\#\mfv/2} +
\sqrt{\lambda}\right)\|f\|_H \\
& & \mbox{} + C \left(\frac{1}{\sqrt{\lambda}} \sum_{\mfv\in\Lambda}
  h_{\mfv}^{\sigma-d/2+\#\mfv/2} + 1\right)\left(\sum_{\mfv\in\Lambda}\frac{1}{\#X_{\mfv}}
\|f_2\|_{\ell_2(X_{\mfv})}^2\right)^{1/2}.
\end{eqnarray*}
Noting
\[
\sum_{\mfv\in \Lambda} \frac{1}{\#X_{\mfv}}
\|f_2\|_{\ell_2(X_\mfv)}^2 \le \sum_{\mfv\in \Lambda}
\frac{1}{\#X_\mfv} \# X_\mfv \|f_2\|_{\ell_\infty(X_\mfv)}^2 \le
\#\Lambda \|f_2\|_{\ell_\infty([\aa,\bb])}^2,
\]
the first stated bound follows immediately. Using also $\sum_{\mfv\in\Lambda}
h_\mfv^{\sigma-d/2+\#\mfv/2} \le C h_{X_\Lambda^{(d)}}^{\sigma-d/2}$, implies the second stated
bound. 
\end{proof}

\begin{remark}
  The final constant in the error bound (\ref{funderror1}) depends on the space dimension
  $d$ in two ways. On the one hand, we have by (\ref{dimconstant}) the
  factor $\prod_{j=1}^d (b_j-a_j)$ which might grow exponentially or
  decrease exponentially. In the most relevant cases where
  $b_j-a_j=1$, it is just one. On the other hand, the remaining generic
  constants mainly depend on the number of elements $\#\Lambda$ in our
  downward closed set. Typically, $\Lambda=\{\mfu\subseteq\mfD : \#\le
  n\}$ with $n$ much smaller than $d$, meaning that $f_1$ contains
  only terms with at most $n$ 
  variables. The cardinality of this set is $\binom{d}{n}$ and
  hence grows at most polynomially in $d$.
  \end{remark}

Finally, the error bound (\ref{funderror1}) and the fact that it only
holds for data sets $X_\Lambda^{(d)}$ requires us to modify
the general approximation approach outlined after Corollary
\ref{cor:genapproxresult} as follows. To approximate $f\in
H^\sigma([\aa,\bb])$ we now proceed as follows. 
\begin{enumerate}
\item Choose $\Lambda\subseteq\cP(\mfD)$ such that $f_1\in
  H_\Lambda^\sigma([\aa,\bb])$ satisfies
  $C \|f-f_1\|_{L_\infty([\aa,\bb])} <\epsilon/2$,
\item Choose the sampling points $X_\Lambda^{(d)}$ such that
  $Ch_{X_\Lambda^{(d)}}^{\sigma-d/2}<\epsilon/2$.
\end{enumerate}
Meaning that we reverse the order of the two approximation
procedures. 

\subsection{Application in Mixed Regularity Sobolev Spaces}

The paper \cite{Rieger-Wendland-24-1} also studies low-term
approximations to mixed regularity Sobolev spaces. The results there
are based on sampling inequalities from \cite{Rieger-Wendland-17-1}.

As usual, The space
  $H_\mix^m(\Omega)$ consists of all function $f\in L_2(\Omega)$
  having weak derivatives $D^\aalpha f\in L_2(\Omega)$ for all
  $\aalpha\in\N_0^d$ with $\|\aalpha\|_\infty\le m$, i.e.  with $\alpha_j\le
  m$, $1\le j\le d$. The space   is equipped with the norm
  \[
  \|f\|_{H_\mix^m(\Omega)}:=\left(\sum_{\|\aalpha\|_\infty\le m}\|D^\aalpha
    f\|_{L_2(\Omega)}^2\right)^{1/2}.
    \]
Again, this definition can be extended to define fractional order
mixed regularity Sobolev spaces $H^{\sigma}_\mix(\Omega)$ if $\Omega$
has a Lipschitz boundary or $\Omega=\R^d$. Again, it is
well-known that these spaces are reproducing kernel Hilbert spaces
whenever $\sigma>1/2$ and that a kernel can be constructed by
restricting a reproducing kernel of 
$H^\sigma_\mix(\R^d)$ to $\Omega$ if an equivalent norm is used.

In this subsection, we will restrict ourselves to $\Omega=[-1,1]^d$ and
hence write $[-1,1]^\mfu$ for $\Omega_\mfu$, $\mfu\subseteq\mfD$. The
following result on $\Lambda$-subspaces can be found in
\cite{Rieger-Wendland-24-1}.

\begin{lemma} Let $\Lambda\subseteq\cP(\mfD)$ be downward closed. Let
  $\sigma>1/2$ Let $H^\sigma_{\mix,\Lambda}([-1,1]^d)$ be the set of
  all functions $f\in H^\sigma_\mix([-1,1]^d)$ having a
  $\Lambda$-representation (\ref{flambda}). Then,
  $H^\sigma_{\mix,\Lambda}([-1,1]^d)$ is closed sub-space of
  $H^\sigma_\mix([-1,1]^d)$.
\end{lemma}

To introduce the required sampling inequality, we will use
{\em sparse grids} based on Clenshaw-Curtis points.  Let 
$n_1=1$ and $n_j=2^{j-1}+1$  for $j\ge 2$. Let $\Lambda\subseteq
\cP(\mfD)$ be downward closed. 
Recall that the {\em Clenshaw-Curtis points}  are the extremal points of the Chebyshev
polynomials given by $Y_{1}= \{0\}$ and 
\[
Y_j = \left\{x_i^{(j)}=-\cos\left(\pi\frac{i-1}{n_{j}-1}\right)
:  1\le i\le n_{j}\right\}, \qquad j>1.
\]
\begin{definition}
For $\emptyset\ne\mfu\subseteq\mfD$ and $q\in\N$ with $q\ge n:=\#\mfu$, the {\em
  sparse grid} $\widetilde{X}_{\mfu,q}\subseteq [-1,1]^\mfu$ based on
the points $\{Y_j\}$ is defined as
\[
\widetilde{X}_{\mfu,q} = \bigcup_{\substack{\ii\in\N^n\\|\ii|=q}}
Y_{\mfu_{i_1}}\times \cdots \times Y_{\mfu_{i_n}}.
\]
In the case of $\mfu=\emptyset$ we set $\widetilde{X}_{\emptyset,q} =
Y_q$. The points of these low-dimensional sets are again extended using the anchor
  $\cc\in[-1,1]^d$, yielding again a point sets $X_{\mfu,q}\subseteq[-1,1]^d$.

For $\qq=\{q_\mfu : \mfu\in \Lambda\}$ with $q_\mfu \ge \#\mfu$,
  a {\em sampling point set   for mixed regularity Sobolev 
  $\Lambda$-functions} is given by
  \begin{equation}\label{mixedpoints}
X_{\Lambda,\qq}^{(d)}=\bigcup_{\mfu\in\Lambda} X_{\mfu,q_\mfu}.
\end{equation}
\end{definition}

In the case of sparse grids, the mesh norm and mesh ratio do not make
sense, as the sparse grids are deliberately sparse in certain
directions. Consequently, sampling inequalities are in terms of the
number of points $N$ of the sparse grid or expressed using the
parameter $\qq$ that defines the sparse grid.

The next result is the required sampling inequality for functions
from $H^\sigma_{\mix,\Lambda}([-1,1]^d)$. It is  \cite[Theorem
  5.6]{Rieger-Wendland-24-1} and it is formulated only for {\em order-$n$-functions}.

\begin{theorem}\label{thm:samplingLambdaSobmixed}
Let  $\sigma>1/2$. Let $1\le n\le d/2$ and $\Lambda=\{\mfu\subseteq
\mfD: \#\mfu \le n\}$. Assume there is a 
$q\in\N$ such that the elements of  $\qq=\{q_\mfu :
\mfu\in\Lambda\}$ have the form $q_\mfu = \#\mfu+q$. 
Let $X=X_{\Lambda,\qq}^{(d)}\subseteq[-1,1]^d$ be the sampling point set
in (\ref{mixedpoints}). Then, there is a constant  $C=C_{\sigma,n}>0$, such
that, with $\tau=\sigma-1/2$,
\[
\|f\|_{L_\infty([-1,1]^d)} \le  C d^n 
\left[d^{\tau n}(\log N)^{\rho_1(\sigma,n)} N^{-\tau}\|f\|_{H^\sigma_\mix([-1,1]^d)} + 
    (\log N)^{\rho_2(n)} \|f\|_{\ell_\infty(X_{\Lambda,\qq}^{(d)})}\right],
\]
for all $f\in H_{\mix,\Lambda}^\sigma([-1,1]^d)$,
where  $\rho_1(\sigma,d) = (\sigma+5/2)(d-1)+1$ and
$\rho_2(d)=2d-1$.
\end{theorem}
Using this sampling inequality together with the generic error bound
from Corollary \ref{cor:genapproxresult}, i.e. with $F_1(N)=(\log
N)^{\rho_1(\sigma,n)} N^{-\sigma+1/2}$ and $F_2(N)=(\log
N)^{\rho_2(n)}$ immediately yields the following result.

\begin{corollary}\label{cor:samplingLambdaSobmixed}
Let the assumptions of Theorem \ref{thm:samplingLambdaSobmixed}
hold. For a function $f\in H:=H^\sigma_\mix([-1,1]^d)$ let $f_1\in
H_1:=H_{\mix,\Lambda}^\sigma([-1,1]^d)$ be the orthogonal projection and
$f_2=f-f_1$. If $Q_{X,\lambda} f$ denotes the penalized least-squares
approximation to $f$ from $H_1$ with the specifically chosen parameter $\sqrt{\lambda}=(\log
N)^{\rho_1(\sigma,n)-\rho_2(n)} N^{-\sigma+1/2}$ based on the grid
$X=X_{\Lambda,\qq}^{(d)}$ then
  \[
  \|f-Q_{X,\lambda} f\|_{L_\infty([-1,1]^d)} \le \widetilde{C}\left[
    (\log N)^{\rho_1} N^{-\sigma+1/2} \|f\|_H + \sqrt{N}(\log
    N)^{\rho_2}\|f_2\|_{L_\infty([-1,1]^d)}\right],
  \]
  where $\rho_1=\rho_1(n,\sigma)$ and $\rho_2=\rho_2(n)$ are defined in
  Theorem \ref{thm:samplingLambdaSobmixed}.
\end{corollary}

\section{Weighted Norms and Their Connection to the Efficient
  Dimension}\label{sec:owen}

The goal of the last section was to to study functions $f$ from 
$H^\sigma([\aa,\bb])$ or $H^\sigma_\mix([\aa,\bb])$ which we have split
in  the form $f=f_1+f_2$,  where $f_1$ is the orthogonal projection
to $H^\sigma_\Lambda([\aa,\bb])$ and
$H^\sigma_{\mix,\Lambda}([\aa,\bb])$, respectively. We will now study
decompositions of the form
\[
f = \sum_{\mfu\in\Lambda} f_{\mfu;\cc} + \sum_{\mfu\in\complement
  \Lambda} f_{\mfu;\cc}=:f_\Lambda+f_{\complement \Lambda}
\]
again with a downward closed set $\Lambda\subseteq\cP(\mfD)$ and its
complement $\complement\Lambda:=\cP(\mfD)\setminus\Lambda$, 
under the assumption that
$\|f_{\complement\Lambda}\|_{L_\infty([\aa,\bb])}$ is small. To this end,
will derive methods of bounding the components
$f_{\mfu;\cc}$ in certain weighted norms, allowing us to conclude
quantitative statements about the size of $\|f_{\complement\Lambda}\|_{L_2([\aa,\bb])}$.

For any multi-index $\aalpha\in\N_0^d$, we have $\|\aalpha\|_\infty
\le \|\aalpha\|_1 \le d \|\aalpha\|_\infty$, which immediately leads
to the inclusions
\[
H^\sigma_\mix([\aa,\bb]) \subseteq H^\sigma([\aa,\bb]) \subseteq
H^{\sigma/d}_\mix([\aa,\bb]) \subseteq H^1_\mix([\aa,\bb])\subseteq C([\aa,\bb]),
\]
where the penultimate inequality obviously only holds if $\sigma>d$. Of course, in
the case of standard Sobolev spaces, there is a loss in smoothness if
some of the variables are fixed. However, for the applications we have
in mind, the functions are actually very smooth such that there is no
real restriction coming from smoothness. Moreover, it is known that
for ANOVA decompositions, the components are often smoother than the
original function, see \cite{Griebel-etal-12-1} and, as mentioned
in the introduction, there is an intrinsic relation between the ANOVA and the
anchored components. In any case, we can restrict ourselves to the
space 
\[
H^1_{\mix} ([\aa,\bb]):= \{f\in L_2([\aa,\bb]) : D^\aalpha f\in
L_2([\aa,\bb]) \mbox{ for all } \aalpha\in\N_0^d  \mbox{ with }
\|\aalpha\|_\infty\le 1\},
\]
which is equipped with the inner product
\begin{equation}\label{hmixinner}
\langle
f,g\rangle_{H^1_{\mix}([\aa,\bb])}:=\sum_{\|\aalpha\|_{\infty}\le
  1} \langle D^\aalpha f,D^\aalpha g \rangle_{L_2([\aa,\bb])} 
\end{equation}
This  space is indeed a Hilbert space of functions,   isomorphic to the tensor
product $\bigotimes_{j=1}^{d} H^{1}([a_j,b_j])$. 
It is important to see that we can write the norm on this space also
in the following form. For any subset $\mfu\subseteq
\mfD=\{1,\ldots,d\}$ with $\#\mfu\le d$ elements, we set
\[
D^\mfu f:= \frac{\partial^{\# \mfu} f}{\prod_{j\in\mfu}\partial_j
}=D^\aalpha f,
\]
where $\aalpha\in\N_0^d$ is the multi-index with entries $\alpha_j=1$
for $j\in\mfu$ and $\alpha_j=0$ for $j\in\mfD\setminus\mfu$.

Since there is an obvious bijective relation between multi-indices
$\aalpha\in\N_0^d$ with $\|\aalpha\|_\infty\le 1$ and the subsets
$\mfu$ of $\mfD$, we thus have
\[
\|f\|_{H^1_\mix([\aa,\bb])}^2 = \sum_{\|\aalpha\|_\infty\le 1}
\|D^{\aalpha} f\|_{L_2([\aa,\bb])}^2 = \sum_{\mfu\subseteq \mfD}
\|D^\mfu f\|_{L_2([\aa,\bb])}^2.
\]
We will soon introduce another weighted but equivalent norm but first need
the following auxiliary result, which is an adaption of the classical Poincare
inequality and is in the same spirit as \cite[Lemma 4.2]{Owen-19-1}
for the ANOVA decomposition. 

\begin{lemma}\label{lem:poincare}
  Let $\mfv \subsetneq \mfu \subseteq \mfD$ and let $\sb\in \R^{\#\mfu}$ be given. 
  Assume $g\in H^1_{\mix}([\aa_{\mfu},\bb_{\mfu} ])$ satisfies $g(\xx)=0$ 
  whenever there is an index $\ell \in\mfu \setminus \mfv$ such that
  $x_\ell=s_\ell$. Then, 
\begin{equation}\label{poincare}
\|D^{\mfv}g\|_{L_{2}([\aa_{\mfu}, \bb_{\mfu} ])} \le C_{\mfv,\mfu}(\aa,\bb)
    \left\|  D^{ \mfu}  g \right\|_{L_{2}([\aa_{\mfu}, \bb_{\mfu} ])},
 \end{equation}
 where the constant $C_{\mfv,\mfu}(\aa,\bb)$ is given by
 \[
   C_{\mfv,\mfu}(\aa,\bb) = \prod_{j\in
     \mfu}(b_j-a_j)^{1/2}\prod_{j\in\mfu\setminus\mfv} (b_j-a_j)^{1/2}
   \le \prod_{j\in \mfu} (b_j-a_j).
   \]
  Obviously, the inequality (\ref{poincare}) also holds for
  $\mfv=\mfu$ for any $g\in H^1_\mix([\aa,\bb])$ with
  $C_{\mfv,\mfu}=1$.
\end{lemma}
\begin{proof}
  Let $k:=|\mfu|$ and recast everything as $k$-dimensional problem.
  The case $\mfv=\mfu$ is trivial. For the general case
  $\mfv\subsetneq \mfu$ and $\ell\in\mfu\setminus \mfv$ we will
assume $g \in C^{\infty}( [\aa_{\mfu},
    \bb_{\mfu}])$. The general case then follows by a standard density argument.
Before we start our argument we first note that for 
For  $1\le j \le k$ with $j\neq \ell$ we have
\begin{align*}
  &\frac{\partial }{\partial y_{j}} g(y_1, \dots y_{\ell-1}, s_{\ell}
  , y_{\ell+1},\dots, y_k)\\ 
  &= \lim_{h\to 0} \frac{g((y_1, \dots y_{\ell-1}, s_{\ell} ,
    y_{\ell+1},\dots, y_k)+h \ee_j ) - g(y_1, \dots y_{\ell-1},
    s_{\ell} , y_{\ell+1},\dots, y_k)}{h}  
  = 0.
\end{align*}
Repeating this argument yields for any set $\mfw\subsetneq
\{1,\ldots,k\}$ and $\ell\not\in\mfw$, 
\begin{equation}\label{an2}
D^{\mfw}g(y_1, \dots y_{\ell-1}, s_{\ell} , y_{\ell+1},\dots, y_k) = 0.
\end{equation}	
We will use this iteratively now, starting with $\ell=k$ and
$\mfw=\emptyset$, i.e. with the standard assumption $g(x_1,\ldots, x_{k-1},s_k)=0$, showing
\begin{eqnarray*}
  g(x_1,\ldots,x_k) & = & g(x_1,\ldots,x_{k-1},s_k)+
  \int_{s_k}^{x_k}\frac{\partial}{\partial x_k}
  g(x_1,\dots, x_{k-1},y_k) d y_k \\
  &= &  \int_{s_k}^{x_k}\frac{\partial}{\partial x_k}
  g(x_1,\dots, x_{k-1},y_k) d y_k, 
\end{eqnarray*}
by the fundamental theorem of calculus. Next, we repeat the argument
using the fundamental theorem of calculus again but this time also  (\ref{an2}) with
$\mfw=\{k\}$ and $\ell={k-1}$ to derive 

\begin{eqnarray*}
  \frac{\partial }{\partial x_k} g(x_1,\ldots,x_{k-1},y_k) &=&
  D^{\{k\}} g(x_1,\ldots,x_{k-2},s_{k-1},y_k)\\
  && \mbox{} +
   \int_{s_{k-1}}^{x_{k-1}} D^{\{k-1,k\}}
   g(x_1,\ldots,x_{k-2},y_{k-1},y_k) dy_{k-1}\\
   & = &    \int_{s_{k-1}}^{x_{k-1}} D^{\{k-1,k\}}
   g(x_1,\ldots,x_{k-2},y_{k-1},y_k) dy_{k-1},
\end{eqnarray*}
yielding altogether
\begin{eqnarray*}
   g(x_1,\ldots,x_k) & = &  \int_{s_k}^{x_k}\int_{s_{k-1}}^{x_{k-1}} D^{\{k-1,k\}}
   g(x_1,\ldots,x_{k-2},y_{k-1},y_k) dy_{k-1} dy_k\\
   & = & \int_{[\sb_{\{ k-1,k\}}, \xx_{\{k-1,k\}}]} D^{\{k-1,k\}} 
		 g(x_1,\dots, s_{k-2},y_{k-1},y_k)  d \yy_{\{k-1,k\}} .
\end{eqnarray*}
Iterating this process eventually shows
\begin{equation}\label{rep1}
  g(\xx) =  \int_{[\sb, \xx]} D^{\{1,\ldots,k\}} g(\yy)  d \yy
  \end{equation}
This allows us to prove the statement for the case
$\mfv=\emptyset$. Here, we have 
\begin{eqnarray*}
  \|g\|^2_{L_{2}([\aa_{\mfu}, \bb_{\mfu} ])} &=& 
  \int_{[\aa_{\mfu}, \bb_{\mfu} ]} \left|g(\xx)\right|^2 d \xx
  =\int_{[\aa_{\mfu}, \bb_{\mfu} ]} \left| \int_{[\sb, \xx]} D^{\{1,\ldots,k\}}
  g(\yy)  d \yy  \right|^2 d \xx\\
  &\le& \int_{[\aa_{\mfu}, \bb_{\mfu} ]} \left(\int_{[\sb, \xx]}\left| D^{\{1,\ldots,k\}} 
  g(\yy)  d \yy \right| \right)^2 d \xx\\
  &\le &\int_{[\aa_{\mfu}, \bb_{\mfu} ]} \left(\int_{[\aa_{\mfu}, \bb_{\mfu} ]}
  \left|  D^{\{1,\ldots,k\}} g(\yy)  d \yy \right| \right)^2 d \xx\\
  &\le& \left( \prod_{j\in \mfu} (b_j-a_j) \right)^{2}
  \left\| D^{\{1,\ldots,k\}} g \right\|^2_{L_{2}([\aa_{\mfu}, \bb_{\mfu} ])} ,
\end{eqnarray*}
where we have used the Cauchy-Schwarz inequality in the last step.

For a general $\mfv \subsetneq \mfu$, we use the above argument
leading to (\ref{rep1}) only for the variables in  $\mfu\setminus
\mfv$ instead of all the variables. Thus, instead of (\ref{rep1}) we obtain
\[
D^{\mfv}g(\xx)= \int_{[\sb_{\mfu\setminus \mfv}, 
    \xx_{\mfu\setminus \mfv} ]}  
D^{\{1,\ldots,k \}}g(y_1,\dots,y_k)  d \yy_{\mfu\setminus \mfv}
\]
and thus
\[
\|D^{\mfv}g\|^2_{L_{2}([\aa_{\mfu}, \bb_{\mfu} ])} \le \prod_{j\in \mfu} (b_j-a_j) 
\prod_{j\in \mfu\setminus \mfv} (b_j-a_j) 
\left\| D^{\{1,\ldots,k\}} g \right\|^2_{L_{2}([\aa_{\mfu}, \bb_{\mfu} ])},
\]
which finishes the proof in the general case.
\end{proof}

After this preparation, we will define the alternative inner product on
$H^1_\mix([\aa,\bb])$, depending on certain weights. The motivation
for this weighted norm comes from, for example,
\cite{Dick-etal-04-1,Sloan-Wozniakowski-98-1, Kuo-etal-12-1, Sloan-etal-04-1},
where  similar weighted norms are used though they are often based on an ANOVA
decomposition rather than an anchored decomposition.

To introduce
it, we recall  \cite[Theorem 3.10]{Rieger-Wendland-24-1} which shows that for any
function  $f\in H^1_{\mix}([\aa,\bb])$ and  any $\mfu\subseteq\mfD$, we
have  $f((\cdot;\cc)_\mfu)\in H^1_{\mix}([\aa_\mfu,\bb_\mfu])$.
Even more, the restriction is continuous, i.e. there is a constant $C(\mfu,\cc,\aa,\bb)>0$ such
that
\begin{equation}\label{equivalence1}
  \|f((\cdot;\cc)_\mfu)\|_{H^1_{\mix}([\aa_\mfu,\bb_\mfu])} \le
  C(\mfu,\aa,\bb,\cc) \|f\|_{H^1_{\mix}([\aa,\bb])}, 
\end{equation}
holds for all $ f\in H^1_{\mix}([\aa,\bb])$. Thus, the following bilinear form is well-defined.

\begin{definition} For every $\mfu\subseteq\mfD$ let $\gamma_\mfu>0$
  be given weights. Then, for $f,g\in H^1_{\mix}([\aa,\bb])$ the
    $\gamma$-weighted inner product is defined by
\begin{eqnarray*}
\left\langle f,g \right\rangle_{H^1_{\cc;\ggamma}([\aa,\bb] )}&:=&
\sum_{\mfu \subseteq \mfD } \gamma^{-1}_{\mfu}
\int_{[\aa_{\mfu},\bb_{\mfu}]} 
D^\mfu  f((\xx;\cc)_{\mfu})
D^\mfu g((\xx;\cc))_{\mfu})
d \xx_{\mfu}\\
& = & \sum_{\mfu \subseteq\mfD} \gamma_\mfu^{-1} \langle D^\mfu
f((\cdot;\cc)_\mfu), D^\mfu g((\cdot;\cc)_\mfu)\rangle_{L_2([\aa_\mfu,\bb_\mfu])}.
\end{eqnarray*}
It induces the norm
\[
\|f\|_{H^1_{\cc;\ggamma}([\aa,\bb])}^2:=\left\langle  f,f
  \right\rangle_{H^1_{\cc;\ggamma}([\aa,\bb])} =
  \sum_{\mfu\subseteq\mfD} \gamma_\mfu^{-1}\|D^\mfu f((\cdot;\cc)_\mfu)\|_{L_2([\aa_\mfu,\bb_\mfu])}^2.
\]
\end{definition}
It is our goal to show that this is indeed an inner product on
$H^1_\mix([\aa,\bb])$ and that the induced norm
is equivalent to the standard norm.
We can now show the well-definedness of the weighted inner product and
the equivalence of norms on $H^1_\mix([\aa,\bb])$. 

\begin{theorem}\label{thm:equivalence}
  The map $\langle
  \cdot,\cdot\rangle_{H^1_{\cc;\ggamma}[\aa,\bb]}\to\R$ is
  well-defined and defines an inner product on $H^1_\mix([\aa,\bb])$.
  The induced norm can alternatively be written as
  \begin{equation}\label{repalternative}
  \|f\|_{H^1_{\cc;\ggamma}([\aa,\bb])}^2 = \sum_{\mfu\subseteq\mfD}
  \gamma_{\mfu}^{-1} \|D^\mfu f_{\mfu;\cc}\|_{L_2([\aa_\mfu,\bb_\mfu])}^2.
  \end{equation}
  There are constants $c,C>0$ such that the induced norm satisfies
  \begin{equation}\label{equivalence2}
c  \|f\|_{H^1_\mix[\aa,\bb]} \le   \|f\|_{H^1_{\cc;\ggamma}([\aa,\bb])} \le C
  \|f\|_{H^1_\mix[\aa,\bb]}, \qquad f\in H^1_\mix([\aa,\bb]).
    \end{equation}
\end{theorem}

\begin{proof}
Obviously, $\langle\cdot,\cdot\rangle_{H^1_{\cc;\ggamma}([\aa,\bb])}$
is bilinear and symmetric. It is well-defined as we have
 \begin{eqnarray*}
    \|f\|_{H^1_{\cc;\ggamma}([\aa,\bb])}^2 & = &
    \sum_{\mfu\subseteq\mfD}\gamma_\mfu^{-1} \|D^\mfu
    f((\cdot;\cc)_\mfu)\|_{L_2([\aa_\mfu,\bb_\mfu])}^2
     \le 
    \sum_{\mfu\subseteq\mfD}\gamma_\mfu^{-1} \|
    f((\cdot;\cc)_\mfu)\|_{H^1_{\mix}([\aa_\mfu,\bb_\mfu])}^2\\
    &\le &
    \sum_{\mfu\subseteq\mfD} \gamma_{\mfu}^{-1} C(\mfu,\aa,\bb,\cc)
    \|f\|_{H^1_{\mix}([\aa,\bb])}^2,
 \end{eqnarray*}
 using (\ref{equivalence1}),   which also shows
the upper bound of (\ref{equivalence2}). To see (\ref{repalternative}), we proceed as
 follows
  If $\mfv \subsetneq \mfu$, the derivative 
$ D^{\mfu} f((\xx;\cc)_{\mfv})$ will vanish,
as we differentiate in at least one direction, in  which  in
$f((\xx;\cc)_{\mfv})$ is constant.
Thus, from the representation (\ref{eq:anova2}) we find
  \[
 D^\mfu f_{\mfu;\cc}(\xx_\mfu) = \sum_{\mfv\subseteq\mfu}
 (-1)^{|\mfu|-|\mfv|} D^\mfu f((\xx;\cc)_\mfv) = D^\mfu
 f((\xx;\cc)_\mfu).
 \]
 This shows immediately the alternative representation given in
 (\ref{repalternative}). The latter representation can now also be
 used to show the lower bound in the norm equivalence
 (\ref{equivalence2}), which in turn also yields definiteness of the
 inner product. To show this lower bound, we first use Lemma
 \ref{lem:poincare} with $g=f_{\mfu;\cc}$. By the annihilation
 property from Theorem  \ref{thm:decomposition} we know that it
 satisfies the assumptions from Lemma \ref{lem:poincare}. Hence, for
 any $\mfv\subseteq\mfu$ we have
 \[
 \|D^\mfv f_{\mfu;\cc}\|_{L_2[\aa_\mfu,\bb_\mfu]} \le
 C_{\mfv,\mfu}(\aa,\bb) \|D^\mfu
 f_{\mfu;\cc}\|_{L_2([\aa_\mfu,\bb_\mfu])}.
 \]
 Thus,  we find
 \begin{eqnarray*}
   \|f\|_{H^1_\mix([\aa,\bb])}^2 & = & \sum_{\mfv\subseteq\mfD}\|D^\mfv
   f\|_{L_2([\aa,\bb])}^2  
    =  \sum_{\mfv\subseteq\mfD} \left\|D^\mfv \sum_{\mfu\subseteq
      \mfD} f_{\mfu;\cc}\right\|_{L_2([\aa,\bb])}^2\\
    & \le & \sum_{\mfv\subseteq\mfD}\left(\sum_{\mfu\subseteq\mfD}
    \|D^\mfv f_{\mfu;\cc}\|_{L_2([\aa,\bb])}\right)^2\\
    & \le & 2^d
    \sum_{\mfv\subseteq\mfD}\sum_{\mfu\subseteq\mfD}\|D^\mfv
    f_{\mfu;\cc}\|_{L_2([\aa,\bb])}^2\\
    & \le & 2^d \sum_{\mfv\subseteq\mfD}\sum_{\mfu\subseteq \mfD}
    \prod_{j\in \mfD\setminus\mfu} (b_j-a_j) \|D^\mfv
    f_{\mfu;\cc}\|_{L_2([\aa_\mfu,\bb_\mfu])}^2.
 \end{eqnarray*}
 Next, we observe that if $\mfv\not\subseteq\mfu$ there is an index
 $j\in\mfv$ with $j\not\in\mfu$. As $f_{\mfu;\cc}$ only depends on the
 variables with index in $\mfu$, its derivative with respect to this
 $j$ is zero, meaning $D^\mfv f_{\mfu;\cc}=0$. Thus, we can proceed

 \begin{eqnarray*}
   \|f\|_{H^1_\mix([\aa,\bb])}^2
   & \le & 2^d \sum_{\mfu\subseteq\mfD} \prod_{j\in \mfD\setminus\mfu} (b_j-a_j)
   \sum_{\mfv\subseteq \mfu} \|D^\mfv f_{\mfu;\cc}\|_{L_2([\aa_\mfu,\bb_{\mfu}])}^2\\
     &  \le& 2^d \sum_{\mfu\subseteq\mfD}
      \prod_{j\in \mfD\setminus\mfu} (b_j-a_j) \sum_{\mfv\subseteq\mfu}
      C_{\mfv,\mfu}(\aa,\bb)\|D^\mfu
      f_{\mfu;\cc}\|_{L_2([\aa_\mfu,\bb_\mfu])}^2\\
      & \le & \widetilde{C}\sum_{\mfu\subseteq\mfD} 
      \|D^\mfu 
      f_{\mfu;\cc}\|_{L_2([\aa_\mfu,\bb_\mfu])}^2\\
      &\le& \widetilde{C} \left(\max_{\mfu\subseteq\mfD}
      \gamma_\mfu\right)\sum_{\mfu\subseteq\mfD} \gamma_\mfu^{-1}
      \|D^\mfu 
      f_{\mfu;\cc}\|_{L_2([\aa_\mfu,\bb_\mfu])}^2\\
      & = & \widetilde{C} \left(\max_{\mfu\subseteq\mfD}
      \gamma_\mfu\right)\|f\|_{H^1_{\cc;\ggamma}([\aa,\bb])}^2,
 \end{eqnarray*}
 where
 \[
 \widetilde{C} = 2^d\max_{\mfu\subseteq\mfD}
 \prod_{j\in\mfD\setminus\mfu}(b_j-a_j)\sum_{\mfv\subseteq\mfu}
 C_{\mfv,\mfu}(\aa,\bb).
 \]
 and where we have used the alternative representation
 (\ref{repalternative}) in the last step.
 \end{proof}

\subsection{Bounds on the Anchored Terms}

While the norm equivalence bounds in Theorem \ref{thm:equivalence} are
extremely coarse, as they have only been derived for the sole purpose
of showing equivalence, the techniques involved allow us to derive
more meaningful bounds on the norm of the terms in the anchored
decomposition.  

The next statement is in this spirit and is also analogue to
\cite[Theorem 4.3]{Owen-19-1} for the ANOVA decomposition.

\begin{theorem}\label{thm:lowerbound}
  Let $C_{\mfv,\mfu}(\aa,\bb)$ be the constants from Lemma
  \ref{lem:poincare}. Define
 \[
\widetilde{C}_\mfu=\widetilde{C}_\mfu(\aa,\bb):=\frac{1}{\sum_{\mfv\subseteq\mfu}
  C_{\mfv,\mfu}(\aa,\bb)^2}, \qquad \mfu\subseteq\mfD.
\]
Then, for any   $f\in H^1_{\mix}([\aa,\bb]) $ the lower bound
  \[
\left\|f\right\|^2_{H^1_{\cc;\ggamma}([\aa,\bb] )} \ge
\sum_{\mfu\subseteq\mfD}\gamma_\mfu^{-1} \widetilde{C}_{\mfu}
\|f_{\mfu;\cc}\|_{H^1_{\mix}([\aa_\mfu,\bb_\mfu])}^2, 
\]
holds.
\end{theorem}
\begin{proof}
 
We recall the alternative representation (\ref{repalternative}) 
\begin{equation}
  \left\|f\right\|^2_{H^1_{\cc;\ggamma}([\aa,\bb] )}=
  \sum_{\mfu\subseteq\mfD}  \gamma_\mfu^{-1}\|D^\mfu
  f_{\mfu;\cc}\|_{L_2([\aa_\mfu,\bb_\mfu])}^2.\label{est1} 
\end{equation}
and that we have by Lemma
\ref{lem:poincare}, 
\[
\|D^{\mfv}f_{\mfu;\cc} \|^2_{L_{2}([\aa_{\mfu}, \bb_{\mfu} ])} \le 
C_{\mfv,\mfu}(\aa,\bb)^2
\left\| D^{\mfu} f_{\mfu;\cc} \right\|^2_{L_{2}([\aa_{\mfu}, \bb_{\mfu} ])},
\]
for $\mfv\subseteq \mfu$,  and thus 
\begin{eqnarray*}
  \|f_{\mfu;\cc}\|^2_{H^1_{\mix}([\aa_{\mfu}, \bb_{\mfu} ])} 
  & =  & \sum_{\mfv \subseteq\mfu} \|D^\mfv
  f_{\mfu;\cc}\|_{L_2([\aa_\mfu,\bb_\mfu])}^2 \\
  &\le& \left(\sum_{\mfv\subseteq\mfu}C_{\mfv,\mfu}(\aa,\bb)^{-2}\right)
\left\| D^{\mfu} f_{\mfu;\cc} \right\|^2_{L_{2}([\aa_{\mfu}, \bb_{\mfu} ])},
\end{eqnarray*}
Using this as a lower bound on $\left\| D^{\mfu} f_{\mfu;\cc}
\right\|^2_{L_{2}([\aa_{\mfu}, \bb_{\mfu} ])}$ and inserting this into
(\ref{est1}) yields

\[
\left\|f\right\|^2_{H^1_{\cc;\ggamma}([\aa,\bb] )} \ge
\sum_{\mfu\subseteq\mfD}\gamma_\mfu^{-1}
\frac{1}{\sum_{\mfv\subseteq\mfu} C_{\mfv,\mfu}(\aa,\bb)^2}
\|f_{\mfu;\cc}\|_{H^1_{\mix}([\aa_\mfu,\bb_\mfu])}^2, 
\]
which is the stated inequality.
\end{proof}

\begin{remark}
In the most popular situations of $[\aa,\bb]=[0,1]^d$ and
$[\aa,\bb]=[1/2,1/2]^d$, we have 
$C_{\mfv,\mfu}(\aa,\bb)=1$ and hence $\widetilde{C}_\mfu = 2^{-\#\mfu
}$ such that the above bound becomes
\[
\left\|f\right\|^2_{H^1_{\cc;\ggamma}([\aa,\bb] )} \ge
\sum_{\mfu\subseteq\mfD}\gamma_\mfu^{-1} 2^{-\#\mfu}
\|f_{\mfu;\cc}\|_{H^1_{\mix}([\aa_\mfu,\bb_\mfu])}^2, 
\]
\end{remark}




\begin{corollary}\label{cor:anchorbound}
The anchored components of a function $f\in H^1_\mix([\aa,\bb])$
satisfy the bound
\[
\|f_{\mfu;\cc}\|^2_{H^1_{\mix}([\aa_{\mfu}, \bb_{\mfu} ])}   \le 
C_\mfu(\aa,\bb) \gamma_\mfu  \|f\|^2_{H^1_{\cc;\ggamma}([\aa, \bb])},
\]
where
\[
C_\mfu(\aa,\bb) = \widetilde{C}_\mfu^{-1} = \sum_{\mfv\subseteq\mfu}
C_{\mfv,\mfu}(\aa,\bb)^2 \le 2^{\#\mfu} \prod_{j\in \mfu} (b_j-a_j)^2.
\]
\end{corollary}
\begin{proof}
For  $f\in H^1_\mix([\aa,\bb])$ we obtain, using Theorem \ref{thm:lowerbound}, 
\begin{eqnarray*}  
\|f_{\mfu;\cc}\|^2_{H^1_{\mix}([\aa_{\mfu}, \bb_{\mfu} ])}   & \le &
\frac{\gamma_\mfu}{\widetilde{C}_\mfu} \gamma_\mfu^{-1}
\widetilde{C}_\mfu \|f_{\mfu;\cc}\|_{H^1_\mix([\aa_\mfu,\bb_\mfu])}^2 
 \le  
C_{\mfu}(\aa,\bb)
\gamma_{\mfu}\left\|f\right\|^2_{H^1_{\cc;\ggamma}([\aa,\bb] )} \\
& \le & \frac{\gamma_\mfu}{\widetilde{C}_\mfu}
\sum_{\mfv\subseteq\mfD} \gamma_\mfv^{-1}\widetilde{C}_\mfv
\|f_{\mfv;\cc}\|_{H^1_\mix([\aa_\mfv,\bb_\mfv])}^2
 =   \frac{\gamma_\mfu}{\widetilde{C}_\mfu} \|f\|^2_{H^1_{\cc;\ggamma}([\aa, \bb])}.
\end{eqnarray*}
\end{proof}

After this, we are able to give a first bound on the $L_\infty$-norm
of $f_2$.

\begin{corollary}\label{cor:final}
For any $f\in H^1_\mix([\aa,\bb])$ and any downward closed set
$\Lambda\subseteq\cP(\mfD)$, the error bound
\[
\|f-f_\Lambda\|_{L_\infty([\aa,\bb])} \le C_{emb} \left(
\sum_{\mfu\in\complement \Lambda} \widetilde{K}_{\mfu}(\aa,\bb)
\gamma_\mfu^{1/2}  \right)
\|f\|_{H_{\cc;\ggamma}^1([\aa,\bb])}.
\]
If $0\le b_j-a_j\le 1$ for $1\le j\le d$, this bound simplifies to
\[
\|f-f_\Lambda\|_{L_\infty([\aa,\bb])} \le
C_{emb}\left(\sum_{\mfu\in\complement \Lambda} 2^{\#\mfu/2}
\gamma_\mfu^{1/2} \right) \|f\|_{H_{\cc;\ggamma}^1([\aa,\bb])}.
\]
In any case, $C_{emb}$ is the constant of the continuous embedding
$H^1_\mix([\aa,\bb])\to L_\infty([\aa,\bb])$ and
\[
\widetilde{K}_\mfu(\aa,\bb) =\left(\prod_{j\in\mfD\setminus \mfu} (b_j-a_j)\right)^{1/2}
  2^{\frac{\#\mfu}{2}}\left(\prod_{j\in\mfu}(b_j-a_j)\right) \
\]
\end{corollary}
\begin{proof}
Using the Sobolev embedding theorem for $H^1_\mix([\aa,\bb])$, the
same ideas that have led to 
(\ref{constantintegral}) and Corollary \ref{cor:anchorbound} yield now
\begin{eqnarray*}
  \lefteqn{ \|f-f_\Lambda\|_{L_\infty([\aa,\bb])}}\\
  & \le & C_{emb}
  \|f-f_\Lambda\|_{H^1_\mix([\aa,\bb])} 
   \le  C_{emb} \sum_{\mfu\in \complement\Lambda}
  \|f_{\mfu;\cc}\|_{H^1_\mix([\aa,\bb])}\\
  & \le & C_{emb} \sum_{\mfu\in \complement\Lambda}
  \left(\prod_{j\in\mfD\setminus \mfu} (b_j-a_j)\right)^{1/2}
  \|f_{\mfu;\cc}\|_{H^1_\mix([\aa_\mfu,\bb_\mfu])}\\
  & \le & C_{emb} \sum_{\mfu\in \complement\Lambda}
  \left(\prod_{j\in\mfD\setminus \mfu} (b_j-a_j)\right)^{1/2}
  2^{\frac{\#\mfu}{2}}\left(\prod_{j\in\mfu}(b_j-a_j)\right) \gamma_\mfu^{1/2}
  \|f_{\mfu;\cc}\|_{H^1_\mix([\aa_\mfu,\bb_\mfu])}.
\end{eqnarray*}
\end{proof}

\section{Application to Parametric Partial Differential Equations}

In this section, we focus on the setting of \cite{Kuo-etal-12-1} in
dealing with parametric PDEs. Let $\dom\subseteq \R^{\ell}$ for $\ell\in
\{1,2,3\}$ be a bounded Lipschitz domain, which we will call the
spatial domain. To avoid confusion between the spatial domain $\dom$
and $\mfD=\{1,\ldots,d\}$, we will only use $\{1,\ldots,d\}$ for the
latter throughout this section. The parameter domain is given by $\Omega =
[-1/2,1/2]^d$ with possibly large $d\in\N$. In principle, even
$d=\infty$ is possible, see the discussion in \cite[Theorem
  5.1]{Kuo-etal-12-1} for the additionally induced error by
truncating the dimension in this form.

For the elliptic PDE under consideration in this section, the diffusion coefficient
will be {\em finite noise}, i.e. it is of the form
\begin{equation}\label{diffusion}
	a(\xx,\yy):=\bar{a}(\xx) + \sum_{j=1}^{d} y_j \psi_{j}(\xx),
        \qquad (\xx,\yy)\in\dom\times\Omega,
\end{equation}
with  given functions $\bar{a},\psi_j \in L_{\infty}(\dom)$ for $1\le j\le
d$ and we consider the following weak problem. For every $\yy\in
\Omega$, we seek the {\em weak solution} $\uu(\cdot,\yy)\in H^1_0(\dom)$ satisfying
\begin{equation}\label{weakPDE}
\int_{\dom} a(\xx,\yy) \nabla_{\xx} u(\xx,\yy) \cdot \nabla_{\xx}
v(\xx) d \xx=F(v), \qquad  v\in H^{1}_{0}(\dom). 
\end{equation}

The following statement of \cite[Theorems 3.1 \&
  4.2]{Kuo-etal-12-1} concerning the weak solution of (\ref{weakPDE})
with diffusion coefficient of the form (\ref{diffusion}) summarizes
what we need here. Note that the cited results even hold in the case $d=\infty$.

\begin{theorem}\label{thm:existenceuniqunessandboundsforu}
  Let $\bar{a},\psi_j \in L_{\infty}(\dom)$ for $1\le j \le d$. The
  assumption $\sum_{j=1}^{d} \|\psi_{j}\|_{L_{\infty}(\dom)}<\infty$
  is clearly satisfied as $d<\infty$. Let the parameter field satisfy  
\[
0<a_{\min} \le a(\xx,\yy) \le a_{\max}, \qquad (\xx,\yy)\in \dom \times \Omega.
\]
For every $\yy\in \Omega$ there is a unique $u(\cdot,\yy) \in H^{1}_{0}(\dom)$
satisfying the weak equation (\ref{weakPDE}).

Moreover, for every $\yy\in \Omega$ and $\mfv\subseteq\{1,\ldots,d\}$, there is the  bound 
\[
\left\| \frac{\partial^{\#\mfv}}{\partial\yy_{\mfv}}u(\cdot,\yy)
\right\|_{H^{1}_{0}(\dom)}\le \frac{\|F\|_{H^{-1}(\dom)}}{a_{\min}}  
\#\mfv ! \prod_{j\in \mfv} \left( \frac{\|\psi_{j}
  \|_{L_{\infty}(\dom)}}{a_{\min}}\right).
\]
\end{theorem}

Next, we extend the mixed regularity Sobolev space $H_\mix^1(\Omega)$
and the weighted mixed regularity Sobolev space
$H_{\ggamma}^1(\Omega):=H_{\00;\ggamma}^1(\Omega)$ to the Bochner
spaces $H_\mix^1(\Omega;H_0^1(\dom))$ and 
$H_{\ggamma}^1(\Omega;H_0^1(\dom))$, respectively. As indicated in the
notation, throughout this section, we will use the anchor $\cc=\00$.

As the former space is only a special case of the latter space with weights
$\gamma_\mfv=1$, we concentrate on the latter. Here, the norm is 
given by
\begin{eqnarray*}
  \|w\|_{H_\ggamma^1(\Omega;H_0^1(\dom))}^2&:=&\sum_{\mfv\subseteq\{1,\ldots,d\}}
  \frac{1}{\gamma_\mfv}
  \int_{\left[-\frac{1}{2};\frac{1}{2} \right]^{\#\mfv}} \left\|
  \frac{\partial^{|\mfv|}}{\partial\yy_{\mfv}}w(\cdot,(\yy;\00)_\mfv)
  \right\|^2_{H^{1}_{0}(\dom)} d \yy_{\mfv} \\
  &= & \sum_{j=1}^\ell \int_\dom \left\|\frac{\partial}{\partial x_j}
  w(\xx,\cdot)\right\|_{H_\ggamma^1(\Omega)}^2 d\xx.
\end{eqnarray*}

With this notation at hand, we can proceed as follows. By Theorem
\ref{thm:existenceuniqunessandboundsforu} we know that the weak
solution satisfies $u\in H_\ggamma^1(\Omega;H_0^1(\dom))=H_\ggamma^1(\Omega)\otimes
H_0^1(\dom)$. As a matter of fact, the solution is infinitely smooth
with respect to the parameter $\yy\in \Omega$. Thus, we can 
bound the  weighted Bochner norm, using the estimate from Theorem
\ref{thm:existenceuniqunessandboundsforu}  to derive
\begin{equation}\label{Bochnerbound1}
\|u\|_{H_\ggamma^1(\Omega;H_0^1(\dom))} \le
\frac{\|F\|_{H^{-1}(\dom)}}{a_{\min}} \left(\sum_{\mfv \subseteq
  \{1,\dots,d\}} \frac{1}{\gamma_\mfv}
  \left( \#\mfv ! \prod_{j\in \mfv} \left( \frac{\|\psi_{j}
    \|_{L_{\infty}(\dom)}}{a_{\min}}\right)\right)^2
  \right)^{\frac{1}{2}}. 
\end{equation}

Finally, in the context of uncertainty quantification, one often is
not particularly interested in the  function $u$ itself but in a derived
quantity. This can be described, using a linear bounded functional
$G\in H^{-1}(\dom)$, i.e. one is interested in the function
$u_G:\Omega\to \R$, defined by $u_G(\yy):=G(u(\cdot,\yy))$, $\yy\in\Omega$. 
The above considerations show, that we have

\begin{equation}\label{uqbound1}
  \|u_G\|_{H^1_\ggamma(\Omega)}\le \|G\|_{H^{-1}(\dom)}
  \frac{\|F\|_{H^{-1}(\dom)}}{a_{\min}} \left(\sum_{\mfv \subseteq
    \{1,\dots,d\}} \frac{1}{\gamma_\mfv} 
		\left( \#\mfv ! \prod_{j\in \mfv} \left(
                \frac{\|\psi_{j}
                  \|_{L_{\infty}(\dom)}}{a_{\min}}\right)\right)^2\right)^{\frac{1}{2}}. 
\end{equation}

Next, turning to approximations of the functions $u$ and $u_G$, we
proceed as in the previous sections. 
For a downward closed subset $\Lambda\subseteq \cP(\{1,\ldots,d\})$ we
let $u_\Lambda$  be the 
$\Lambda$-approximations of $u$, i.e.
\[
u_\Lambda(\xx,\yy) = \sum_{\mfv\in \Lambda}
u(\xx,\cdot)_\mfv(\yy_\mfv),
\qquad (\xx,\yy)\in\dom\times\Omega,
\]
and $u_{G,\Lambda} (\yy):= G(u_\Lambda(\cdot,\yy))$, $\yy\in\Omega$.
Then, Corollary \ref{cor:final} yields
\begin{eqnarray*}
  \|u(\cdot,\yy)-u_\Lambda(\cdot,\yy)\|_{H_0^1(\dom)}^2 & = &
  \sum_{j=1}^\ell \int_\dom \left|\frac{\partial}{\partial
    x_j}\left(u(\xx,\yy)-u_\Lambda(\xx,\yy)\right)\right|^2 d x \\
  &\le &\sum_{j=1}^\ell \int_\dom C_{emb}^2 \left(\sum_{\mfv\in
    \complement \Lambda} 2^{\#\mfv/2} \gamma_\mfv^{1/2} \right)^2
  \left\|\frac{\partial}{\partial
    x_j}u(\xx,\cdot)\right\|_{H_\ggamma^1(\Omega)}^2 d \xx\\
  & = & C_{emb}^2 \left(\sum_{\mfv\in
    \complement \Lambda} 2^{\#\mfv/2} \gamma_\mfv^{1/2}
  \right)^2\|u\|_{H_\ggamma^1(\Omega;H_0^1(\dom))}^2,
\end{eqnarray*}
which shows the next result.
\begin{corollary}\label{cor:uqbound1}
Let $u\in H_\ggamma^1(\Omega;H_0^1(\dom))$ be the weak solution of the
parametric PDE and let $u_\Lambda$ be the $\Lambda$-term approximation
of $u$ using a downward closed subset
$\Lambda\subseteq\cP(\{1,\ldots,d\})$. Let $G\in H^{-1}(\dom)$
describe the quantity of interest. Then,
\begin{eqnarray*}
\|u-u_\Lambda\|_{L_\infty(\Omega;H_0^1(\dom))} &\le&
C_{emb}\sum_{\mfv\in \complement\Lambda} 2^{\#\mfv/2}\gamma_\mfv^{1/2} 
\|u\|_{H_\ggamma^1(\Omega;H_0^1(\dom))},\\
\|u_G- u_{G,\Lambda}\|_{H_\ggamma^1(\Omega)} & \le &
\|G\|_{H^{-1}(\dom)} C_{emb}\sum_{\mfv\in \complement\Lambda} 2^{\#\mfv/2}\gamma_\mfv^{1/2} 
\|u\|_{H_\ggamma^1(\Omega;H_0^1(\dom))},
\end{eqnarray*}
where the norm $\|u\|_{H_\ggamma^1(\Omega;H_0^1(\dom))}$ can 
be bounded by (\ref{Bochnerbound1}).
\end{corollary}

To bound this further, we need to specify the weights $\gamma_\mfv$ in
such a way that we can bound the two terms on the right-hand side,
i.e. essentially the term
\[
\left(\sum_{\mfv\in \complement\Lambda}
2^{\#\mfv/2}\gamma_\mfv^{1/2}\right) \left(\sum_{\mfv \subseteq
  \{1,\dots,d\}} \frac{1}{\gamma_\mfv}
  \left( \#\mfv ! \prod_{j\in \mfv} \left( \frac{\|\psi_{j}
    \|_{L_{\infty}(\dom)}}{a_{\min}}\right)\right)^2
  \right)^{\frac{1}{2}}. 
  \]
  This will be done in the next sub-section.
  
  \subsection{Bounds for Specific Weights}
  
Again, we follow \cite{Kuo-etal-12-1} to define and analyze the
weights.

\begin{definition} Let $\mfU\subseteq \cP(\N)$ contain all finite subsets of
  $\N$, i.e. $\mfu\subseteq\N$ belongs to  $\mfU$ if and only if
  $\#\mfu<\infty$. For given weights 
  $\alpha_j>0$, $j\in\N$ define $ \gamma: (1/2,1] \times \N_0 \times \N
\times \mfU \to (0,\infty)$ via 
\begin{equation}\label{eq:defgamma}
\gamma(c,n,m,\mfu):=\gamma_{\mfu}(c;(n,m)):= \left( \frac{
  ( \#\mfu+n) !}{m} \prod_{j \in \mfu}
\frac{\alpha_{j}}{\sqrt{\rho(c)}}\right)^{\frac{2}{1+c}}, 
\end{equation}
where $\rho: (\frac{1}{2},1] \to (0,\infty)$ is defined by
$\rho(c) =\frac{2\zeta(2c)}{(2\pi^2)^{c}}$,
using the classical $\zeta$-function.
\end{definition}
For $n=0$ and $m=1$, this corresponds to the choice in
\cite[Eq. (6.3)]{Kuo-etal-12-1}. For $n=3$ and $m=6$, this corresponds
to the choice of \cite[Eq. (65)]{Kuo-etal-15-1}.

\begin{lemma}\label{lem:sequence}
Let $c\in \left(1/2,1\right]$ and $n\in\N_0$, $m\in \N$ be given. Let
$\alpha_j>0$, $j\in \N$,  be a sequence such that 
\begin{equation}\label{eq:condalpha}
  \sum_{j\in \N } \alpha_{j}^{\frac{1}{1+c}} \in \left(0,
  \frac{\rho(c)^{\frac{1}{2(1+c)}}}{2} \right). 
\end{equation}
Let   $\Lambda:=\left\{\mfu \in \mfU : \#\mfu\le L \right\}$ contain
all subsets of $\N$ with at most $L$ elements. Then, for the product and order
dependent weights given in \eqref{eq:defgamma}, there is a constant
$C_2(n)$ depending only on $n$ such that
\begin{equation}\label{eq:epsilonbound}
  \sum_{\mfu\in \complement\Lambda}2^{\#\mfu}\gamma^{\frac{1}{2}}_{\mfu}(c;(n,m)) \le 
  C_{2}(n) m^{-\frac{1}{c+1}} \frac{\left( \sum_{j\in \N} \left(
    \frac{2^{c+1}}{\sqrt{\rho(c)}} \alpha_{j}\right)
    ^{\frac{1}{1+c}}\right)^{L+1} }{1- \sum_{j\in \N}  
    \left( \frac{2^{c+1}}{\sqrt{\rho(c)}} \alpha_{j} \right) ^{\frac{1}{1+c}} }
\end{equation}
in particular, we have $\sum_{\mfu\in
  \complement\Lambda}2^{\#\mfu}\gamma^{\frac{1}{2}}_{\mfu}(c;(n,m))
\to 0$ as $L\to \infty$. 
\end{lemma}

\begin{proof}
The following direct computations are in the spirit of \cite[Lemma 6.3]{Kuo-etal-12-1}.
We start with the equality
\begin{align*}
  2^{\#\mfu}\gamma^{\frac{1}{2}}_{\mfu}(c;(n,m))
  &=\left( \frac{ ( \#\mfu+n) !}{m} \prod_{j \in \mfu}
  \frac{ 2^{c+1}\alpha_{j}}{\sqrt{\rho(c)}}\right)^{\frac{1}{1+c}}.
\end{align*}
Next, introducing $\mfU_k:=\{\mfu\in \mfU : \#\mfu = k\}$ and summing this up, yields	
\begin{align*}
  \sum_{\mfu\in \complement\Lambda}
  2^{\#\mfu}\gamma^{\frac{1}{2}}_{\mfu}(c;(n,m)) &=
  \sum_{\substack{k\in \N\\k>L}} \sum_{\mfu \in
    \mfU_k}2^{\#\mfu}\gamma^{\frac{1}{2}}_{\mfu}(c;(n,m))
=
  \sum_{k>L} \sum_{\mfu \in \mfU_k}\left( \frac{ ( \#\mfu+n) !}{m}
  \prod_{j \in \mfu} \frac{
    2^{c+1}\alpha_{j}}{\sqrt{\rho(c)}}\right)^{\frac{1}{1+c}}\\ 
  &= m^{-\frac{1}{c+1}}\sum_{k\in \N} \sum_{\mfu \in
    \mfU_k} \left(( \#\mfu+n) !\right)^{\frac{1}{1+c}} 
  \left( \prod_{j \in \mfu} \frac{
    2^{c+1}\alpha_{j}}{\sqrt{\rho(c)}}\right)^{\frac{1}{1+c}}\\ 
  &= m^{-\frac{1}{c+1}}\sum_{k\in \N}\left(( k+n)
  !\right)^{\frac{1}{1+c}}\left(\frac{2^{c+1}}{\sqrt{\rho(c)}}
  \right)^{\frac{k}{1+c}} \sum_{\mfu \in \mfU_k}  
  \left( \prod_{j \in \mfu} \alpha_{j}\right)^{\frac{1}{1+c}}.
\end{align*}
Next, we want to identify subsets $\mfu =
\{u_1,\ldots,u_k\} \in\mfU_k$  with vectors
$\uu=(u_1,\ldots,u_k)\in\N^k$. However, since a set is not ordered and
does not contain the same element multiple times, we proceed as
follows. We define
\(
\widetilde{\N}^k := \left\{\uu\in \N^k : u_j\ne u_i \mbox{ for } i\ne j
\right\},
\)
i.e. the subset of $\N^k$ containing only vectors with distinct
components. Then, each element $\uu=(u_1,\ldots,u_k)\in
\widetilde{\N}^k$ defines a set $\mfu=\{u_1,\ldots,u_k\}\in
\mfU_k$. Obviously, two vectors $\uu,\vv\in\widetilde{\N}^k$ 
define the same set $\mfu\in \mfU_k$ if there is a permutation, i.e. a
bijective map $\ssigma:\{1,\ldots,k\}\to\{1,\ldots,k\}$, such that
$\uu=\vv\circ \ssigma$. If $S_k$ denotes the group of all such
permutations, we have  a natural
equivalence relation on $\widetilde{\N}^k$ by defining  $\uu\sim\vv$ if there
is a $\ssigma\in S_k$ such that $\uu=\vv\circ \ssigma$. If we denote
corresponding equivalence classes  by $[\uu]=\{\vv\in \widetilde{\N}^k:
\vv\sim\uu\}$ then  each comprises $\#[\uu]=k!$ elements. Obviously,
we now have a  bijection $U_k: (\widetilde{\N}^k /\sim)  \to \mfU_k$,
defined by $[\uu] \mapsto \{u_1,\ldots,u_k\}  = \mfu \in \mfU_k$. 
Making  use of this, we can proceed as follows:
\begin{eqnarray*}
\lefteqn{\sum_{k>L} \sum_{\mfu \in
 \mfU_k}2^{\#\mfu}\gamma^{\frac{1}{2}}_{\mfu}(c;(n,m))}\\
&=& m^{-\frac{1}{c+1}}\sum_{k>L}\left(( k+n)
!\right)^{\frac{1}{1+c}}\left(\frac{2^{c+1}}{\sqrt{\rho(c)}}
\right)^{\frac{k}{1+c}} \sum_{\uu\in (\widetilde{\N}^k/\sim)}  
\left( \prod_{j =1}^{k} \alpha_{U_{k}(\uu)_{j}}\right)^{\frac{1}{1+c}}\\
&=& m^{-\frac{1}{c+1}}\sum_{k>L}\left(( k+n)
!\right)^{\frac{1}{1+c}}\left(\frac{2^{c+1}}{\sqrt{\rho(c)}}
\right)^{\frac{k}{1+c}} \frac{1}{k!}\sum_{\uu \in \N^k}  
\left( \prod_{j =1}^{k} \alpha_{u_{j}}\right)^{\frac{1}{1+c}}\\
&=&m^{-\frac{1}{c+1}}\sum_{k>L}\left(( k+n)
!\right)^{\frac{1}{1+c}}\left(\frac{2^{c+1}}{\sqrt{\rho(c)}}
\right)^{\frac{k}{1+c}} \frac{1}{k!}\left( \sum_{j\in \N}
\alpha^{\frac{1}{1+c}}_{j} \right)^{k}\\ 
&=&m^{-\frac{1}{c+1}}\sum_{k>L}\left(\frac{( k+n)
  !}{k!}\right)^{\frac{1}{1+c}} \left(
\frac{1}{k!}\right)^{\frac{c}{1+c}}\left( \sum_{j\in \N} \left(
\frac{2^{c+1}}{\sqrt{\rho(c)}} \alpha_{j} \right) ^{\frac{1}{1+c}}
\right)^{k} .
\end{eqnarray*}
Observing  that $c\in (\frac{1}{2},1]$ means particularly
\[
\frac{1}{1+c} \in \left(\frac{1}{2},\frac{2}{3} \right] \quad \text{and} \quad \frac{c}{1+c} \in \left[\frac{1}{4},\frac{1}{3} \right),
\]
leading to 
\[
\left(\frac{( k+n) !}{k!}\right)^{\frac{1}{1+c}} \left(
\frac{1}{k!}\right)^{\frac{c}{1+c}} \le C_1(n) k^{\frac{2n}{3}}
(k!)^{-\frac{1}{3}}= C_1(n) \left(\frac{k^{2n}}{k!}
\right)^{\frac{1}{3}} \le C_2(n). 
\]
This shows 
\[
\sum_{k>L} \sum_{\mfu \in
  \mfU_k}2^{\#\mfu}\gamma^{\frac{1}{2}}_{\mfu}(c;(n,m)) \le C_{2}(n)m^{-\frac{1}{c+1}}\
\sum_{k>L} \left( \sum_{j\in \N} \left(
\frac{2^{c+1}}{\sqrt{\rho(c)}} \alpha_{j} \right) ^{\frac{1}{1+c}}
\right)^{k} .
\]
Finally, employing the well-known identity $\sum_{k=L+1}^{\infty}q^{k}
=\frac{q^{L+1}}{1-q}$ for $q\in (0,1)$ yields 
\begin{align*}
\sum_{\mfu\in \complement\Lambda}2^{\#\mfu}\gamma^{\frac{1}{2}}_{\mfu}(c;(n,m)) &= \sum_{k\in \N} \sum_{\mfu \in \mfU_k}2^{\#\mfu}\gamma^{\frac{1}{2}}_{\mfu}(c;(n,m)\\
&\le C_{2}(n) m^{-\frac{1}{c+1}} \frac{\left( \sum_{j\in \N} \left(
  \frac{2^{c+1}}{\sqrt{\rho(c)}} \alpha_{j}\right)
  ^{\frac{1}{1+c}}\right)^{L+1} }{1- \sum_{j\in \N} \left(
  \frac{2^{c+1}}{\sqrt{\rho(c)}} \alpha_{j} \right) ^{\frac{1}{1+c}} },
\end{align*}
which is the stated inequality.
\end{proof}
After this auxiliary result, we are now in the position to give the
necessary bounds for such weights.

\begin{theorem}\label{thm:bounds}
  Fix $c\in \left(\frac{1}{2},1 \right]$ and choose the weights as
in \eqref{eq:defgamma} with $\alpha_j=\frac{\|\psi_{j}
  \|_{L_{\infty}(\dom)}}{a_{\min}}$, i.e.,
\[
\gamma_{\mfv}=\gamma(c,0,1,\mfv):= \left(  \#\mfv!
\prod_{j \in \mfv} \frac{\|\psi_{j} \|_{L_{\infty}(\dom)}}{
  \sqrt{\rho(c)} a_{\min}}\right)^{\frac{2}{1+c}}. 
\]
If  the assumptions
\begin{equation}\label{eq:condpsi}
\max_{1\le j \le d} \frac{\|\psi_{j}
  \|_{L_{\infty}(\dom)}}{a_{\min}}\le 1
\quad \text{and}  \quad
\sum_{j=1}^{d} \left(\frac{\|\psi_{j}
  \|_{L_{\infty}(\dom)}}{a_{\min}}\right)^{\frac{1}{1+c}} <
\min\left\{\sqrt{6},\frac{\rho(c)^{\frac{1}{2(1+c)}}}{2} \right\} 
\end{equation}
are satisfied, then we have 
\begin{equation}\label{eq:ugfinite}
\|u\|_{H_\ggamma^1(\Omega;H_0^1(\cD))}<\infty
\end{equation}
and for $\Lambda:=\left\{\mfv \subseteq \{1,\dots,d\} :
\#\mfv\le n \right\}$ we have the bound 
\begin{equation}\label{eq:epsilonbound2}
  \sum_{\mfv\in \complement\Lambda}2^{\#\mfv}\gamma^{\frac{1}{2}}_\mfv
  \le  
  C \frac{\epsilon^{n+1}}{1-\epsilon}
\end{equation}
with
\begin{equation}\label{epsilon}
\epsilon:=\frac{2}{\rho(c)^{1/2(1+c)}}\sum_{j=1}^d \left(
\frac{\|\psi_j\|_{L_\infty(\dom)}}{ a_{\min}}\right)^{1/(1+c)} \in (0,1).
\end{equation}
\end{theorem}
\begin{proof}
  For the finiteness of the Bochner norm, we employ \cite[Theorem
    6.4]{Kuo-etal-12-1}.  
  With $b_j=\frac{\|\psi_{j} \|_{L_{\infty}(\dom)}}{a_{\min}}$, we see
  that the condition in \cite[Eq. (6.4)]{Kuo-etal-12-1} is
  trivially satisfied for any $p\in (0,1]$ as $d<\infty$. 
For the condition in  \cite[Eq. (6.5)]{Kuo-etal-12-1}, we note
that for any $c\ge 0$ it holds that 
\begin{equation*}
  \sum_{j=1}^{d}  \frac{\|\psi_{j} \|_{L_{\infty}(\dom)}}{a_{\min}}
  \le \max_{1\le j\le d} \left(\frac{\|\psi_{j}
    \|_{L_{\infty}(\dom)}}{a_{\min}}\right)^{\frac{c}{1+c}}  
  \sum_{j=1}^{d}\left(\frac{\|\psi_{j}
    \|_{L_{\infty}(\dom)}}{a_{\min}} \right)^{\frac{1}{1+c}}\le
  \sqrt{6} 
\end{equation*}
due to the assumptions from \eqref{eq:condpsi}. As we have no
restriction on $p$, we can pick any $c\in \left( \frac{1}{2},1\right]$
  and get \eqref{eq:ugfinite}. The final statement
  \eqref{eq:epsilonbound2}  
	follows directly from Lemma \ref{lem:sequence}. 
\end{proof}

Note that if  (\ref{eq:condpsi}) even holds for $d=\infty$, the
$\epsilon$ defined in (\ref{epsilon}) is
even independent of $d$,  as the involved
sums on the right-hand side converge for $d\to\infty$. 

\begin{corollary}\label{cor:finalcor1}
Let $u\in H_\ggamma^1(\Omega;H_0^1(\dom))$ be the weak solution of the
parametric PDE and let $u_\Lambda$ be the $\Lambda$-term approximation
of $u$ using $\Lambda=\{\mfv \subseteq\cP(\{1,\ldots,d\}): \#\mfv\le
n\}$. Let $G\in H^{-1}(\dom)$  describe the quantity of interest. Then,
\begin{eqnarray*}
\|u-u_\Lambda\|_{L_\infty(\Omega;H_0^1(\dom))} &\le& C \left(\epsilon/\sqrt{2}\right)^{n+1}
\|u\|_{H_\ggamma^1(\Omega;H_0^1(\dom))}\\
\|u_G- u_{G,\Lambda}\|_{H_\ggamma^1(\Omega)} & \le &
C \|G\|_{H^{-1}(\dom)}
\left(\epsilon/\sqrt{2}\right)^{n+1}\|u\|_{H_\ggamma^1(\Omega;H_0^1(\dom))}.
\end{eqnarray*}
\end{corollary}
\begin{proof}
  This immediately follows from Corollary \ref{cor:uqbound1}, Theorem
  \ref{thm:bounds} and 
\[
\sum_{\mfv\in\complement\Lambda} 2^{\#\mfv/2}\gamma_\mfv^{1/2} =
\sum_{\mfv\in\complement\Lambda} 2^{-\#\mfv/2} 2^{\#\mfv/2}
\gamma_{\mfv}^{1/2} \le 2^{-(n+1)/2} \widetilde{C}\epsilon^{n+1},
\]
as $\mfv\in\complement \Lambda$ means $\#\mfv \ge n+1$.
\end{proof}

We are now in the position to combine the results of Section
\ref{sec:mismeasured} and this section, in particular those of
Corollary \ref{cor:samplingLambdaSobmixed} and
Corollary \ref{cor:finalcor1}  to derive a bound on
the approximation error for the quantity of interest $u_G$ of a
solution to a parametric PDE and its discrete approximation
$Q_{X,\lambda}u_G$. However, there are the following things to
consider. First of all, while for Corollary
\ref{cor:samplingLambdaSobmixed} we have assumed the underlying
parameter domain to be $\Omega=[-1,1]$, we require for Corollary
\ref{cor:finalcor1} the domain to be $\Omega=[-1/2,1/2]$. Fortunately,
a scaling argument shows that Corollary
\ref{cor:samplingLambdaSobmixed} also holds for $\Omega=[-1/2,1/2]$.
The involved constants might change but will only depend on $n$ if
$\Lambda = \{\mfu\in \{1,\ldots,d\} : \#\mfu \le n\}$ is used.

More importantly, in Corollary \ref{cor:samplingLambdaSobmixed}, we
have decomposed any $f\in H^\sigma_{\mix}(\Omega)$ into $f=f_1+f_2$, where
$f_1\in H^\sigma_{\mix,\Lambda}(\Omega)$ is the orthogonal projection
of $f$ onto $H^\sigma_{\mix,\Lambda}(\Omega)$. In Corollary
\ref{cor:finalcor1}, however, we rather used the  decomposition of
$f=u_G$ into  $f=f_\Lambda+f_{\complement \Lambda}$, where $f_\Lambda$
uses the anchored components $f_\mfu$ of $f$ for $\mfu\in \Lambda$.
Fortunately, in the situation of mixed regularity spaces, if the
anchored kernels from \cite{Kuo-etal-10-1, Rieger-Wendland-24-1} are used, it follows from
\cite[Theorem 3.10]{Rieger-Wendland-24-1} that $f_\Lambda\in
H^\sigma_{\mix,\Lambda}(\Omega)$. Moreover, by \cite[Proposition
  6.1]{Rieger-Wendland-24-1} it follows that
$H^\sigma_{\mix}(\Omega)=H^\sigma_{\mix,\Lambda}(\Omega)\oplus
H^\sigma_{\mix,\complement\Lambda}(\Omega)$ is an orthogonal
decomposition. This means that we indeed have $f_1=f_\Lambda$ and that
we have the following result.

\begin{corollary}
Let $\Omega=[-1/2,1/2]^d$.
Let $u\in H_\ggamma^1(\Omega;H_0^1(\dom))$ be the weak solution of the
parametric PDE, let  $\Lambda=\{\mfv \subseteq\cP(\{1,\ldots,d\}): \#\mfv\le
n\}$ and let $G\in H^{-1}(\dom)$  define  the quantity of
interest $u_G$. Assume that $u_G\in H^\sigma_{\mix}(\Omega)$ for some
$\sigma\ge 1$. Finally, let $u_{X,G,\Lambda}=Q_{X,\lambda} u_{G}$
denote the regression function from Corollary
\ref{cor:samplingLambdaSobmixed} from the subspace
$H_{\mix,\Lambda}^\sigma(\Omega)$ using the smoothing parameter
$\lambda$ as
described in Corollary \ref{cor:samplingLambdaSobmixed}. Then, we have the bound
\begin{eqnarray}
  \| u_{G} - u_{X,G,\Lambda}\|_{L_{\infty}(\Omega)}& \le & \widetilde{C}
    (\log N)^{\rho_1} N^{-\sigma+1/2}
  \|u_{G}\|_{H^\sigma_\mix(\Omega)} \nonumber \\
  & & \mbox{} +  \widetilde{C} \sqrt{N}(\log
  N)^{\rho_2} (\epsilon/\sqrt{2})^{n+1} |u\|_{H_\ggamma^1(\Omega;H_0^1(\dom))},\label{eq:corfinal2}
  \end{eqnarray}
where $\rho_1=\rho_1(n,\sigma)$ and $\rho_2=\rho_2(n)$ are defined in
 Theorem \ref{thm:samplingLambdaSobmixed}.
\end{corollary}

\begin{proof}
We start with applying Corollary \ref{cor:samplingLambdaSobmixed}
directly. In the notation of that corollary we have $f=u_G$,
$f_1=u_{G,\Lambda}$ and $u_{X,G,\Lambda} = Q_{X,\lambda}u_G$. Thus, we have
\begin{eqnarray*}
\|u_{G,\Lambda}-u_{X,G,\Lambda}\|_{L_\infty(\Omega)} &\le & \widetilde{C}\Big[
    (\log N)^{\rho_1} N^{-\sigma+1/2}
  \|u_{G}\|_{H^\sigma_\mix(\Omega)}\\
  & & \mbox{} + \sqrt{N}(\log
    N)^{\rho_2}\|u_{G} - u_{G,\Lambda}\|_{L_\infty(\Omega)}\Big].
  \end{eqnarray*}
To bound the term $u_G-u_{G,\Lambda}$, we use the Sobolev
embedding theorem, Theorem \ref{thm:equivalence}  and Corollary
\ref{cor:finalcor1} to obtain  
\begin{align*} 
  \| u_{G} - u_{G,\Lambda}\|_{L_{\infty}(\Omega)} &\le C_{emb}
  \| u_{G} - u_{G,\Lambda}\|_{H^1_\mix(\Omega)}\\
  &\le C_{emb} C
  \|u_{G}- u_{G,\Lambda}\|_{H^1_\ggamma(\Omega)} \le
  C_{emb} C \tilde{C} (\epsilon/\sqrt{2})^{n+1} |u\|_{H_\ggamma^1(\Omega;H_0^1(\dom))}.
\end{align*}
\end{proof}
\begin{remark}
  The constant $C$ in the above proof is given in the proof of Theorem  \ref{thm:equivalence}
as
\[ C^2= 2^d\max_{\mfu\subseteq\{1,\ldots,d\}}
 \prod_{j\in\{1,\ldots,d\}\setminus\mfu}(\frac{1}{2}+\frac{1}{2})\sum_{\mfv\subseteq\mfu}
 C_{\mfv,\mfu}(\Omega)  \left(\max_{\mfu\in\complement \Lambda}
      \gamma_\mfu\right) \le 2^{2d} \left(\max_{\mfu\in\complement \Lambda}
      \gamma_\mfu\right), 
      \]
 where $C_{\mfv,\mfu}(\Omega)  =1$ due to Lemma
 \ref{lem:poincare}. Moreover, we point out that the maximum has to be  
considered only over the index set $\complement \Lambda$ as $u_{G} -
u_{G,\Lambda}$ will only have contributions there.  
We obtain from \eqref{eq:epsilonbound2} the bound
\[
  \gamma_\mfu = 2^{-2\#\mfu} \left
  ( 2^{\#\mfu} \gamma^{\frac{1}{2}}_{\mfu}\right)^{2} \le 2^{-2\#\mfu} C^2 \left(
  \frac{\epsilon^{n+1}}{1-\epsilon} \right)^2,
\]
where $\epsilon\in (0,1)$ is defined in \eqref{epsilon}.
\end{remark}

Please note that in \eqref{eq:corfinal2} the cardinality of the
sampling point set  $N$ depends on $n$ as well.

\section*{Acknowledgment}
This work is funded by the Deutsche Forschungsgemeinschaft (DFG, German
Research Foundation) –Projektnummer 452806809.

\end{document}